\documentclass[a4paper,11pt]{article}

\usepackage[T1]{fontenc}
\usepackage{lmodern}

\usepackage{dsfont}

\usepackage[latin1]{inputenc}
\usepackage{amsmath}
\usepackage{amsthm}
\usepackage{amssymb}
\usepackage{mathrsfs}
\usepackage{graphicx}
\usepackage[all]{xy}
\usepackage{hyperref}

\usepackage{stmaryrd}
\usepackage{caption}

\usepackage{abstract} 

\newtheorem{thm}{Theorem}[section]
\newtheorem{cor}[thm]{Corollary}
\newtheorem{claim}[thm]{Claim}
\newtheorem{fact}[thm]{Fact}

\newtheorem{lemma}[thm]{Lemma}
\newtheorem{prop}[thm]{Proposition}

\theoremstyle{definition}
\newtheorem{definition}[thm]{Definition}
\newtheorem{ex}[thm]{Example}
\newtheorem{remark}[thm]{Remark}

\title{Algebraic characterisation of relatively hyperbolic special groups}
\date{\today}
\author{Anthony Genevois}

\begin{document}

\maketitle

\begin{abstract}
This article is dedicated to the characterisation of the relative hyperbolicity of Haglund and Wise's special groups. More precise, we introduce a new combinatorial formalism to study (virtually) special groups, and we prove that, given a cocompact special group $G$ and a finite collection of subgroups $\mathcal{H}$, then $G$ is hyperbolic relative to $\mathcal{H}$ if and only if (i) each subgroup of $\mathcal{H}$ is convex-cocompact, (ii) $\mathcal{H}$ is an almost malnormal collection, and (iii) every non-virtually cyclic abelian subgroup of $G$ is contained in a conjugate of some group of $\mathcal{H}$. As an application, we show that a virtually cocompact special group is hyperbolic relative to abelian subgroups if and only if it does not contain $\mathbb{F}_2 \times \mathbb{Z}$.
\end{abstract}

\tableofcontents

\section{Introduction}

CAT(0) cube complexes are now classical objects in geometric group theory. The explanation is twofold. First, the combinatorics of their hyperplanes provides a powerful tool to study their geometry, allowing us to answer many questions about the algebra and the geometry of groups acting on CAT(0) cube complexes. Next, there exist many different and interesting groups acting on such complexes, providing a large collection of potential applications of the theory. They include for instance Coxeter groups, right-angled Artin groups, one-relator groups with torsion, small cancellation groups, hyperbolic free-by-cyclic groups, and many 3-manifold groups. Consequently, given a group, it is a good strategy to try to make it act on a CAT(0) cube complex.

Nevertheless, the class of groups acting on CAT(0) cube complexes still contains groups with exotic behaviors. One of the most impressive examples is the family of simple groups constructed by Burger and Mozes \cite{BurgerMozes}, which act geometrically on products of two trees. In this article, we focus on groups acting on CAT(0) cube complexes in a very specific way, which allows us to avoid these exotic examples. Namely, we are interested in \emph{special groups}, ie., groups which are fundamental groups of \emph{special cube complexes}. Although the class of special groups is smaller than the class of groups acting geometrically on CAT(0) cube complexes, it is worth noticing that many groups of interest turn out to be (virtually) special, including Coxeter groups and many hyperbolic groups, such as small cancellation groups and hyperbolic 3-manifold groups. The initial motivation for the introduction of special groups in \cite{MR2377497} was to prove separability properties, which play a fundamental role in the proof of the virtual Haken conjecture \cite{MR3104553}. But many other properties of interest can be deduced from being special, typically all the properties which hold for right-angled Artin groups (such as being residually nilpotent, satisfying Tits' alternative or being biorderable) and which are stable under taking subgroups.

In this article, our goal is to construct a new and simple combinatorial formalism to study special groups. The main observation is that, in a special cube complex, a path (with a fixed initial vertex) is uniquely determined by the sequence of oriented hyperplanes it crosses, so that paths may be thought of as words of oriented hyperplanes. Moreover, the homotopy between paths (with fixed endpoints) coincides with the equivalence relation generated by the following elementary transformations: in a word of oriented hyperplanes, remove or add subwords of the form $JJ^{-1}$; if $J_1$ and $J_2$ are transverse, replace the subword $J_1J_2$ with $J_2J_1$. Consequently, the elements of the fundamental group of our special cube complex $X$ can be thought of as words of oriented hyperplanes submitted to the previous natural relations. Technically, we identify the fundamental groupoid of $X$ with a groupoid defined from \emph{legal} words of oriented hyperplanes. In full generality, we do not consider words of hyperplanes, but we introduce what we call a \emph{special coloring}, which colors the oriented hyperplanes of $X$, and next we consider words of colors. However, setting the color set as the set of oriented hyperplanes leads to words of oriented hyperplanes, so it is a good example to keep in mind. 

%As an immediate consequence of the formalism, it follows that a special group embeds into a right-angled Artin group. In particular, one gets a combinatorial proof of the following statement proved in \cite{MR2377497}:
%
%\begin{thm}\label{thm:main1}
%Let $X$ be a special cube complex. If $\Delta$ denotes its crossing graph, then $\pi_1(X)$ embeds into the right-angled Artin group $A(\Delta)$. 
%\end{thm}
%
%Recall that the \emph{crossing graph} of a cube complex is the graph whose vertices are its hyperplanes and whose edges link two hyperplanes whenever they are transverse. We emphasize that the embedding we construct from an arbitrary special coloring (see Theorem \ref{thm:embedRAAG} for a precise statement) may lead to a right-angled Artin group which is ``smaller'' than the right-angled Artin group associated to the crossing graph. See Example \ref{ex:embedRAAG}.

The strength of this new formalism is that it provides a description of the subcomplexes (in the universal cover) which decompose as Cartesian products (the ``non-hyperbolic subspaces''). Such a description allows us to understand the obstructions to diverse hyperbolicities. More precisely, as an application of our formalism, we state and prove algebraic criteria to several hyperbolic properties.

The existence of connections between the algebra and the geometry of a group is a central idea in geometric group theory. Among the first historical illustrations of this idea, one can mention Stallings' theorem stating that a finitely generated group has at least two ends if and only if it splits non-trivially over a finite subgroup; and Gromov's theorem stating that a finitely generated group has polynomial growth if and only if it is virtually nilpotent. In this article, the geometric properties which interest us are Gromov-hyperbolicity and relative hyperbolicity. 

\emph{Gromov-hyperbolic groups}, or \emph{hyperbolic groups} for short, were introduced in \cite{GromovHyp}. Typically, they are the groups which are, when thought of as metric spaces via their Cayley graphs, ``negatively curved''. This class includes free groups, surface groups, more generally fundamental groups of compact riemannian manifolds of negative sectional curvature, and small cancellation groups. Being hyperbolic implies severe restrictions on the algebra of a group (see \cite{GhysHarpe} for more information), but the most emblematic property of hyperbolic groups is that they cannot contain subgroups isomorphic to $\mathbb{Z}^2$. Although the converse does not hold, namely there exist non-hyperbolic groups without $\mathbb{Z}^2$ as a subgroup, it turns out to be satisfied in many cases of interest and several conjectures are dedicated to this problem. For instance, it is unknown whether CAT(0) groups are necessarily hyperbolic if they do not contain $\mathbb{Z}^2$ as a subgroup. Even in the context of CAT(0) cube complexes, this problem remains open, although partial results were found in \cite{SpecialHyp, PeriodicFlatsCC, NTY}. In particular, a positive answer was proved in \cite{SpecialHyp} for (virtually) special groups by studying Davis complexes and their convex subspaces. Thanks to our formalism, we are able to give a simpler and shorter proof of \cite[Corollary 4]{SpecialHyp}. More precisely, we prove:

\begin{thm}\label{thm:main2}
Let $X$ be a compact special cube complex. If its universal cover contain an $n$-dimensional combinatorial flat for some $n \geq 1$, then $\pi_1(X)$ contains $\mathbb{Z}^n$.
\end{thm}

Since the fundamental group of a nonpositively-curved cube complex $X$ is hyperbolic if and only if the universal cover of $X$ does not contain a (combinatorial) two-dimensional flat, it follows that:

\begin{cor}\label{cor:main2}
A virtually cocompact special group is hyperbolic if and only if it does not contain $\mathbb{Z}^2$. 
\end{cor}

The first generalisation of hyperbolic groups, already suggested in \cite{GromovHyp}, is the class of \emph{relatively hyperbolic groups}. Geometrically, a group $G$ is hyperbolic relative to a finite collection of subgroups $\mathcal{H}$, called the \emph{peripheral subgroups}, if the subspaces of $G$ which are not negatively-curved are included in a conjugate of some subgroup of $\mathcal{H}$ and if the conjugates of the subgroups of $\mathcal{H}$ do not fellow-travel (loosely speaking, they do not interact). So we allow some non-negative curvature but we want a control on it. The main motivation for the introduction of relatively hyperbolic groups was to construct a general framework dealing with hyperbolic manifolds of finite volume (in this case, the peripheral subgroups are the fundamental groups of the cusps). Algebraically, the picture is essentially the following: a subgroup of a relatively hyperbolic group either has a behavior similar to hyperbolic groups or it is included in a peripheral subgroups. We refer to \cite{OsinRelativeHyp} for more details on relatively hyperbolic groups. The main criterion we obtain is the following:

\begin{thm}\label{thm:RHspecialMain}
Let $G$ be a cocompact special group and $\mathcal{H}$ a finite collection of subgroups. Then $G$ is hyperbolic relative to $\mathcal{H}$ if and only if the following conditions are satisfied:
\begin{itemize}
	\item each subgroup of $\mathcal{H}$ is convex-cocompact;
	\item $\mathcal{H}$ is an almost malnormal collection;
	\item every non-virtually cyclic abelian subgroup of $G$ is contained in a conjugate of some group of $\mathcal{H}$. 
\end{itemize}
\end{thm}

It is worth noticing that the characterisation of relatively hyperbolic right-angled Coxeter groups (proved in \cite{BHSC, coningoff}), and more generally the characterisation of relatively hyperbolic graph products of finite groups (which is a particular case of \cite[Theorem 8.33]{Qm}), follow easily from Theorem \ref{thm:specialrelhyp}. However, this criterion does not provide a purely algebraic characterisation of relatively hyperbolic special groups, since the subgroups need to be convex-cocompact and that convex-cocompactness is not an algebraic property. Indeed, with respect to the canonical action $\mathbb{Z}^2 \curvearrowright \mathbb{R}^2$, the cyclic subgroup generated by $(0,1)$ is convex-cocompact, whereas the same subgroup is not convex-cocompact with respect to the action $\mathbb{Z}^2 \curvearrowright \mathbb{R}^2$ defined by $(0,1) : (x,y) \mapsto (x+1,y+1)$ and $(1,0) : (x,y) \mapsto (x+1,y)$. On the other hand, we do not know if the convex-cocompacteness assumption can be removed, ie., we know do not know whether or not a finitely generated malnormal subgroup is automatically convex-cocompact. 

Nevertheless, the previous statement provides an algebraic criterion if one restricts our attention to a collection of subgroups which we know to be convex-cocompact. In this spirit, we show the following beautiful characterisation of virtually special groups which are hyperbolic relative to virtually abelian subgroups.

\begin{thm}\label{thm:RHspecialAbelianMain}
A virtually cocompact special group is hyperbolic relative to a finite collection of virtually abelian subgroups if and only if it does not contain $\mathbb{F}_2 \times \mathbb{Z}$ as a subgroup.
\end{thm}

Up to our knowledge, it is the first general algebraic characterisation of relative hyperbolicity in the literature, making. Theorem \ref{thm:RHspecialAbelianMain} the most significative contribution of this paper. (It also generalises Corollary \ref{cor:main2}.)

As an interesting consequence of this theorem it turns out that, among special groups, being hyperbolic relative to abelian subgroups is preserved under elementary equivalence; in fact, the property depends only on the universal theory of the group. Therefore, once we know that limit groups are virtually special, we can reprove that they are hyperbolic relative to abelian subgroups (see Section \ref{section:elementaryequivalence}).

Being hyperbolic relative to free abelian subgroups is also known as \emph{having isolated flats} among CAT(0) groups \cite{IsolatedFlats}, and it is an interesting particular case of relatively hyperbolic groups as it provides interesting additional information on the group. We refer to \cite{DrutuSapirI, DrutuSapirII, GrovesLimitRHI, GrovesLimitRHII} for more information on the tools available to study such groups and the kind of information which can be deduced.

We do not know if the conclusion of Theorem \ref{thm:RHspecialAbelianMain} holds for any group acting geometrically on some CAT(0) cube complex. Probably the question is as hard as determining whether groups acting on CAT(0) cube complexes are hyperbolic if and only if they do not contain $\mathbb{Z}^2$.

As a concluding remark, I would like to emphasize that the present work is inspired by the formalism associated to Guba and Sapir's diagram groups \cite{MR1396957}, which are in my opinion underappreciated. In particular, several results mentioned above are in the same spirit as results I proved in \cite{arXiv:1505.02053, arXiv:1507.01667, article3, Qm} about diagram groups. (However, there is no intersection between the results of these papers and the results proved here.)

The organisation of our paper is as follows. In Section \ref{section:prel}, we give some preliminary definitions and properties about nonpositively-curved cube complexes. In Section \ref{section:formalism}, we describe our combinatorial formalism. A couple of easy applications of the formalism are proved in Section \ref{section:warmup}; in particular, Theorem \ref{thm:main2} is proved in Section \ref{section:flat}. Section \ref{section:RHspecial} is dedicated to the relative hyperbolicity of special groups: more precisely, Theorem \ref{thm:RHspecialMain} is proved in Section \ref{section:RH} and Theorem \ref{thm:RHspecialAbelianMain} in Section \ref{section:RHabelian}.

\paragraph{Acknowledgments.} I am grateful to Peter Ha\"{i}ssinsky for his comments on the first version of this article.

\section{Nonpositively-curved cube complexes}\label{section:prel}

\paragraph{Special cube complexes.} A \textit{cube complex} is a CW-complex constructed by gluing together cubes of arbitrary (finite) dimension by isometries along their faces. Furthermore, it is \textit{nonpositively curved} if the link of any of its vertices is a simplicial \textit{flag} complex (ie., $n+1$ vertices span a $n$-simplex if and only if they are pairwise adjacent), and \textit{CAT(0)} if it is nonpositively curved and simply-connected. See \cite[page 111]{MR1744486} for more information.

\medskip \noindent
A fundamental feature of cube complexes is the notion of \textit{hyperplane}. Let $X$ be a nonpositively curved cube complex. Formally, a \textit{hyperplane} $J$ is an equivalence class of edges, where two edges $e$ and $f$ are equivalent whenever there exists a sequence of edges $e=e_0,e_1,\ldots, e_{n-1},e_n=f$ where $e_i$ and $e_{i+1}$ are parallel sides of some square in $X$. Notice that a hyperplane is uniquely determined by one of its edges, so if $e \in J$ we say that $J$ is the \textit{hyperplane dual to $e$}. Geometrically, a hyperplane $J$ is rather thought of as the union of the \textit{midcubes} transverse to the edges belonging to $J$. See Figure \ref{figure17}.
\begin{figure}
\begin{center}
\includegraphics[trim={0 13cm 10cm 0},clip,scale=0.4]{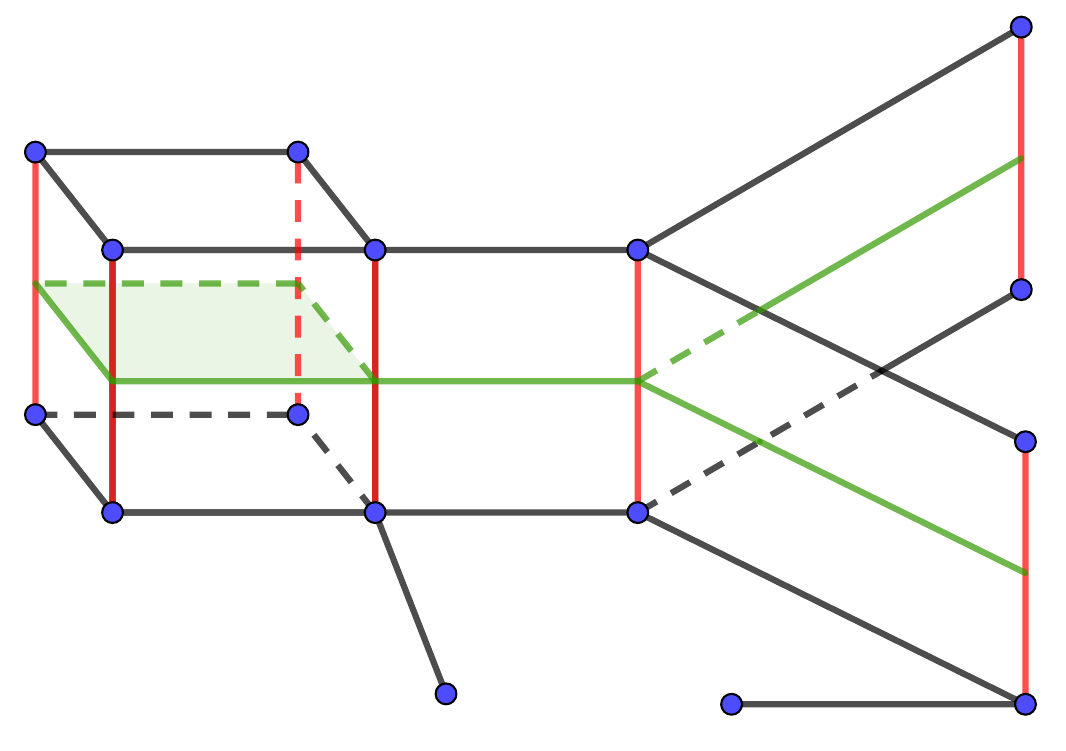}
\caption{A hyperplane (in red) and its geometric realisation (in green).}
\label{figure27}
\end{center}
\end{figure}
Similarly, one may define \textit{oriented hyperplanes} as classes of oriented edges. If $J$ is the hyperplane dual to an edge $e$ and if we fix an orientation $\vec{e}$, we will note $\vec{J}$ the oriented hyperplane dual to $\vec{e}$. It may be thought of as an \textit{orientation} of $J$, and we will note $- \vec{J}$ the opposite orientation of $J$. 

\begin{definition}\label{def:OCrossing}
Let $X$ be a cube complex. The \emph{crossing graph} $\Delta X$ is the graph whose vertices are the hyperplanes of $X$ and whose edges link two transverse hyperplane. Similarly, the \emph{oriented crossing graph} $\vec{\Delta} X$ is the graph whose vertices are the oriented hyperplanes of $X$ and whose edges link two oriented hyperplanes whenever their underlying unordered hyperplanes are transverse.
\end{definition}

\noindent
Roughly speaking, \emph{special cube complexes} are nonpositively-curved cube complexes which do not contain ``pathological configurations'' of hyperplanes. Let us define precisely what these configurations are.

\begin{definition}
Let $J$ be a hyperplane with a fixed orientation $\vec{J}$. We say that $J$ is:
\begin{itemize}
	\item \textit{2-sided} if $\vec{J} \neq - \vec{J}$,
	\item \textit{self-intersecting} if there exist two edges dual to $J$ which are non-parallel sides of some square,
	\item \textit{self-osculating} if there exist two oriented edges dual to $\vec{J}$ with the same initial points or the same terminal points, but which do not belong to a same square.
\end{itemize}
Moreover, if $H$ is another hyperplane, then $J$ and $H$ are:
\begin{itemize}
	\item \textit{transverse} if there exist two edges dual to $J$ and $H$ respectively which are non-parallel sides of some square,
	\item \textit{inter-osculating} if they are transverse, and if there exist two edges dual to $J$ and $H$ respectively with one common endpoint, but which do not belong to a same square.
\end{itemize}
\end{definition}
\noindent
Sometimes, one refers 1-sided, self-intersecting, self-osculating and inter-osculating hyperplanes as \textit{pathological configurations of hyperplanes}. The last three configurations are illustrated on Figure \ref{figure17}.
\begin{figure}
\begin{center}
\includegraphics[trim={0, 19.5cm 1cm 0},clip,scale=0.45]{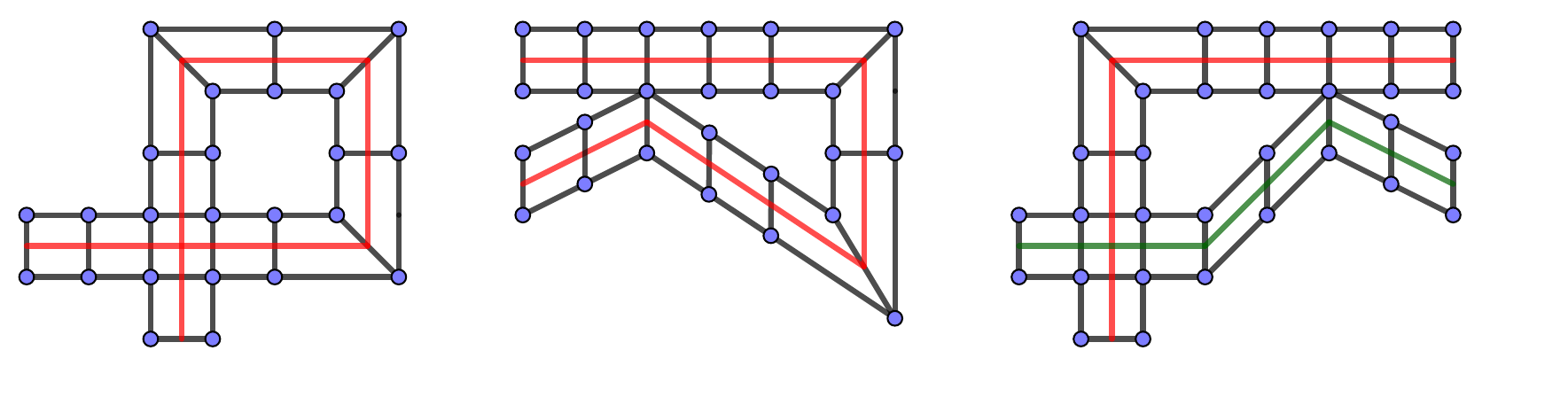}
\caption{From left to right: self-intersection, self-osculation, inter-osculation.}
\label{figure17}
\end{center}
\end{figure}

\begin{definition}
A \emph{special cube complex} is a nonpositively curved cube complex whose hyperplanes are two-sided and which does not contain self-intersecting, self-osculating or inter-osculating hyperplanes. A group which can be described as the fundamental group of a compact special cube complex is \emph{special}. A \emph{virtually special group} is a group which contains a finite-index subgroup which is special. 
\end{definition}

\noindent
We emphasize that, in this article, a special group is the fundamental group of a \emph{compact} special cube complex. To avoid ambiguity with fundamental groups of not necessarily compact special cube complexes, they are sometimes referred to as \emph{cocompact special groups} in the litterature. (We followed this convention in the introduction for clarity.)

\paragraph{Homotopy.} Let $X$ be a cube complex (not necessarily nonpositively-curved). For us, a \emph{path} in $X$ is the data of a sequence of successively adjacent edges. What we want to understand is when two such paths are homotopic (with fixed endpoints). For this purpose, we need to introduce the following elementary transformations. One says that:  
\begin{itemize}
	\item a path $\gamma \subset X$ contains a \emph{backtrack} if the word of oriented edges read along $\gamma$ contains a subword $ee^{-1}$ for some oriented edge $e$;
	\item a path $\gamma' \subset X$ is obtained from another path $\gamma \subset X$ by \emph{flipping a square} if the word of oriented edges read along $\gamma'$ can be obtained to corresponding word of $\gamma$ by replacing a subword $e_1e_2$ with $e_2'e_1'$ where $e_1',e_2'$ are opposite oriented edges of $e_1,e_2$ respectively in some square of $X$. 
\end{itemize}
We claim that these elementary transformations are sufficient to determine whether or not two paths are homotopic. More precisely:

\begin{prop}\label{prop:cubehomotopy}
Let $X$ be a cube complex and $\gamma,\gamma' \subset X$ two paths with the same endpoints. The paths $\gamma,\gamma'$ are homotopic (with fixed endpoints) if and only if $\gamma'$ can be obtained from $\gamma$ by removing or adding backtracks and flipping squares. 
\end{prop}

\noindent
This statement follows from the fact that flipping squares provide the relations of the fundamental groupoid of $X$; see \cite[Statement 9.1.6]{BrownGroupoidTopology} for more details.

\paragraph{Contracting subcomplexes.} In Section \ref{section:RHspecial}, we will need several statements proved in \cite{article3} about \emph{contracting subcomplexes} in CAT(0) cube complexes. Below, we recall basic definitions and state the results which we will use later. 

\begin{definition}
Let $S$ be a metric space. A subspace $C \subset S$ is \emph{contracting} if there exists some constant $d \geq 0$ such that the nearest-point projection onto $C$ of any ball disjoint from $C$ has diameter at most $d$. An isometry $g \in \mathrm{Isom}(S)$ is \emph{contracting} if, for some basepoint $x \in S$, the map $n \mapsto g^n \cdot x$ induces a quasi-isometric embedding $\mathbb{Z} \to S$ whose image is contracting. 
\end{definition}

\noindent
In CAT(0) cube complexes, the nearest-point projection onto a given convex subcomplex has nice properties. For instance, according to \cite[Lemma 1]{arXiv:1505.02053} and \cite[Lemmas 2.8 and 2.10]{coningoff}, one has:

\begin{prop}
Let $\widetilde{X}$ be a CAT(0) cube complex and $C \subset \widetilde{X}$ a convex subcomplex. For every vertex $x \in \widetilde{X}$, there exists a unique vertex of $C$ minimising the distance to $x$. This vertex, denoted by $\mathrm{proj}_C(x)$, is the \emph{projection} of $x$ onto $C$. Moreover, the map $\mathrm{proj}_C : X \to C$ is $1$-Lipschitz, and, for every vertex $x \in \widetilde{X}$, any hyperplane separating $x$ from its projection onto $C$ must separate $x$ from $C$.
\end{prop}

\noindent
The next characterisation of contracting convex subcomplexes in CAT(0) cube complexes was essentially proved in \cite{article3}. Before, we need to introduce several definitions:
\begin{itemize}
	\item A \emph{facing triple} is the data of three hyperplanes such that no one separates the two others.
	\item A \emph{join of hyperplanes} $(\mathcal{H}, \mathcal{V})$ is the data of two collections of hyperplanes $\mathcal{H}, \mathcal{V}$ which do not contain any facing triple and such that any hyperplane of $\mathcal{H}$ is transverse to any hyperplane of $\mathcal{V}$; it is \emph{$L$-thin} for some $L \geq 0$ if $\min( \# \mathcal{H}, \# \mathcal{V}) \leq L$. If moreover $\mathcal{H}, \mathcal{V}$ are collections of pairwise disjoint hyperplanes, $(\mathcal{H}, \mathcal{V})$ is a \emph{grid of hyperplanes}.
	\item A \emph{flat rectangle} is a combinatorial isometric embedding $[0,p] \times [0,q] \hookrightarrow \widetilde{X}$; it is \emph{$L$-thin} for some $L \geq 0$ if $\min(p,q) \leq L$, and \emph{$L$-thick} otherwise. Most of the time, a flat rectangle is identified with its image. 
\end{itemize}
Moreover, given a subset $Y$ in some CAT(0) cube complex, we denote by $\mathcal{H}(Y)$ the set of the hyperplanes separating at least two vertices of $Y$.

\begin{thm}\label{thm:contractingCube}
Let $\widetilde{X}$ be a finite-dimensional CAT(0) cube complex and $C \subset \widetilde{X}$ a convex subcomplex. The following statements are equivalent:
\begin{itemize}
	\item[(i)] $C$ is contracting;
	\item[(ii)] the join of hyperplanes $(\mathcal{H}, \mathcal{V})$ satisfying $\mathcal{H} \cap \mathcal{H}(C)= \emptyset$ and $\mathcal{V} \subset \mathcal{H}(C)$ are uniformly thin;
	\item[(iii)] the grid of hyperplanes $(\mathcal{H}, \mathcal{V})$ satisfying $\mathcal{H} \cap \mathcal{H}(C)= \emptyset$ and $\mathcal{V} \subset \mathcal{H}(C)$ are uniformly thin;
	\item[(iv)] the flat rectangles $R : [0,p] \times [0,q] \hookrightarrow \widetilde{X}$ satisfying $R \cap C = [0,p] \times \{ 0\}$ are uniformly thin.
\end{itemize}
\end{thm}

\begin{proof}
The equivalence $(i) \Leftrightarrow (ii)$ is proved by \cite[Theorem 3.6]{article3}. The implication $(ii) \Rightarrow (iii)$ is clear, and the converse follows from the assumption that $\widetilde{X}$ is finite-dimensional and from Ramsey's theorem (see for instance \cite[Lemma 3.7]{coningoff}). The implication $(iii) \Rightarrow (iv)$ is clear. It remains to prove the converse $(iv) \Rightarrow (iii)$. 

\medskip \noindent
So let $(\mathcal{H}, \mathcal{V})$ be a grid of hyperplanes satisfying $\mathcal{H} \cap \mathcal{H}(C)= \emptyset$ and $\mathcal{V} \subset \mathcal{H}(C)$. Write $\mathcal{V}$ as $\{V_1, \ldots, V_p \}$ such that $V_i$ separates $V_{i-1}$ and $V_{i+1}$ for every $2 \leq i \leq p-1$; and $\mathcal{H}$ as $\{H_1, \ldots, H_q\}$ such that $H_i$ separates $H_{i-1}$ and $H_{i+1}$ for every $2 \leq i \leq q-1$ and such that $H_q$ separates $H_1$ and $C$. Let $D \to \widetilde{X}$ be a disc diagram of minimal complexity bounded by the cycle of subcomplexes $\mathcal{C}= (N(V_1),C,N(V_p),N(H_1))$. (We refer to \cite{coningoff, article3} and references therein for more information on disc diagrams in CAT(0) cube complexes.) According to \cite[Corollary 2.17]{coningoff}, $D \to \widetilde{X}$ is a flat rectangle; write $D=[0,a] \times [0,b]$ such that $[0,a] \times \{0 \} \subset C$, $\{a\} \times [0,b] \subset N(V_p)$, $[0,a] \times \{b\} \subset N(H_1)$ and $\{0\} \times [0,b] \subset N(V_1)$. Suppose that the vertex $(c,d) \in D$ belongs to $C$. Since $C$ is convex, the intervals between $(0,0)$ and $(c,d)$, and $(a,0)$ and $(c,d)$, have to lie in $C$. As a consequence, $[0,a] \times [d,b]$ defines a disc diagram bounded by $\mathcal{C}$. By minimality of $D$, necessarily $d=0$. Thus, we have proved that $D \cap C= [0,a] \times \{0\}$, which implies that there exists a uniform constant $L \geq 0$ such that $\min(a,b) \leq L$. But the hyperplanes $H_2, \ldots, H_q$ separate $H_1$ and $C$, so they must intersect $\{0\} \times [0,b]$, hence $b \geq q-1$; similarly, the hyperplanes $V_2, \ldots, V_{p-1}$ separate $V_1$ and $V_p$, so they must intersect $[0,a] \times \{ 0 \}$, hence $a \geq p-2$. Therefore, $\min(p,q) \leq \min(a,b) +2 \leq L+2$. This proves $(iii)$. 
\end{proof}

\noindent
The previous statement justifies the following definition:

\begin{definition}
Let $\widetilde{X}$ be a CAT(0) cube complexes, $C \subset \widetilde{X}$ a convex subcomplex and $L \geq 0$ a constant. Then $C$ is \emph{$L$-contracting} if every join of hyperplanes $(\mathcal{H}, \mathcal{V})$ satisfying $\mathcal{H} \cap \mathcal{H}(C)= \emptyset$ and $\mathcal{V} \subset \mathcal{H}(C)$ is $L$-thin.
\end{definition}

\noindent
We conclude this section with a statement which will be useful in Section \ref{section:RHspecial} (and which is inspired from \cite[Lemmas 4.4 and 4.5]{MR3339446}). 

\begin{lemma}\label{lem:geodcontracting}
Let $X$ be a CAT(0) cube complex and $C$ a contracting subcomplex. There exists a constant $L \geq 0$ such that, if $x,y \in X$ are two vertices such that the distance between their projections onto $C$ is at least $L$, then, for every geodesic $[x,y]$ between $x$ and $y$, the unique vertex of $[x,y]$ at distance $d(x,C)$ from $x$ must be at distance at most $L$ from the projection of $x$ onto $C$. 
\end{lemma}

\begin{proof}
Let $K \geq 0$ be such that the projection onto $C$ of every ball disjoint from $C$ has diameter at most $K$. We begin by proving the following fact:

\begin{fact}
For all vertices $x,y \in X$ whose projections onto $C$ are at distance at least $2(K+1)$ apart, then 
$$|d(x,y)-d(x,p)-d(p,q)-d(q,y)| \leq 2(K+1)$$
where $p$ and $q$ denote respectively the projections of $x$ and $y$ onto $C$.
\end{fact}

\noindent
Let $x,y \in X$ be two vertices whose respective projections $p,q$ onto $C$ are at least at distance  $2(K+1)$ apart. Fix a geodesic $[x,y]$ between $x$ and $y$. Notice that, if $d(x,y) \leq d(x,p)$, then the open ball $B(x,d(x,y))$ is disjoint from $C$, hence $d(p,q) \leq K+1$, contradicting our assumption. Therefore, $[x,y]$ contains a unique vertex $x'$ at distance $d(x,p)$ from $x$; notice that, since the open ball $B(x,d(x,x'))$ is disjoint from $C$, the distance between the projection of $x'$ onto $C$, say $p'$, and $p$ is at most $K+1$. Similarly, $[x,y]$ contains a unique vertex $y'$ at distance $d(y,q)$ from $y$; moreover, the distance between the projection of $y'$ onto $C$, say $q'$, and $q$ is at most $K+1$. Next, notice that, if $d(x,y') < d(x,x')$, then $y'$ belongs to the open ball $B(x,d(x,x'))$, then $d(p,q') \leq K$, which implies that
$$d(p,q) \leq d(p,q')+d(q',q) \leq K +K+1= 2K+1 < 2(K+1) ,$$
contradicting our assumption. Therefore, $d(x,y') \geq d(x,x')$, which implies that 
$$d(x,y)=d(x,x')+d(x',y')+d(y',y) = d(x,p)+d(x',y')+d(q,y).$$
Because the projection onto $C$ is $1$-Lipschitz, we know that $d(x',y') \geq d(p',q')$. But 
$$d(p',q') \geq d(p,q)-d(p,p')-d(q,q') \geq d(p,q) - 2(K+1).$$
The inequality
$$d(x,y) \geq d(x,p)+d(p,q)+d(q,y)-2(K+1)$$
follows. Finally, the triangle inequality shows that $d(x,y) \leq d(x,p)+d(p,q)+d(q,y)$, which concludes the proof of our fact. 

\medskip \noindent
Now we are ready to prove our lemma. Let $x,y \in X$ be two vertices whose projections onto $C$ are at least at distance $4(K+1)$ apart. Let $p,q \in C$ denote the projections onto $C$ of $x,y$ respectively. By reproducing the beginning of the proof of the previous fact, one shows that there exists a unique pair of vertices $x',y' \in [x,y]$ satisfying $d(x,x')=d(x,C)$ and $d(y,y')=d(y,C)$, and that $d(x,y)=d(x,x')+d(x',y')+d(y',y)$. From the triangle inequality, we know that $d(x,y) \leq d(x,p)+d(p,q)+d(q,y)$, which implies together with the previous inequality, that $d(x',y') \leq d(p,q)$ since $d(x,x')=d(x,p)$ and $d(y,y')=d(y,q)$. 

\medskip \noindent
Let $p',q'$ denote the projections onto $C$ of $x',y'$ respectively. Notice that $d(p',q') \geq 2(K+1)$. Indeed, because the open ball $B(x,d(x,x'))$ is disjoint from $C$, necessarily $d(p,p') \leq K+1$; and similarly $d(q,q') \leq K+1$. Hence 
$$d(p',q') \geq d(p,q)-d(p,p')-d(q,q') \geq 2(K+1).$$
Therefore, our previous fact applies, so that
$$\begin{array}{lcl} d(p,q) & \geq & d(x',y') \geq d(x',p')+d(p',q')+d(q',y')-2(K+1) \\ \\ & \geq & d(x',p)+d(p,q)+d(q,y')-2(K+1) -2d(p,p')-2d(q,q') \\ \\ & \geq & d(x',p)+d(p,q)+d(q,y') - 6(K+1) \end{array}$$
hence $d(x',p) \leq 6(K+1)$. Consequently, $L=6(K+1)$ is the constant we are looking for, concluding the proof of our lemma. 
\end{proof}

\section{Formalism: special colorings}\label{section:formalism}

\noindent
In this section, we introduce the formalism which we will use in the next sections to prove the statements mentioned in the introduction. Our central definition is the following:

\begin{definition}\label{def:specialcolor}
Let $X$ be a cube complex. A \emph{special coloring} $(\Delta,\phi)$ is the data of a graph $\Delta$ and a coloring map 
$$\phi : V \left( \vec{\Delta} X \right) \to V(\Delta) \sqcup V(\Delta)^{-1},$$
where $\vec{\Delta}X$ denotes the oriented crossing graph (see Definition \ref{def:OCrossing}), satisfying the following conditions:
\begin{itemize}
	\item for every oriented hyperplane $J$, the equality $\phi(J^{-1})=\phi(J)^{-1}$ holds;
	\item two transverse oriented hyperplanes have adjacent colors;
	\item no two oriented hyperplanes adjacent to a given vertex have the same color;
	\item two oriented edges with same origin whose dual oriented hyperplanes have adjacent colors must generate a square. 
\end{itemize}
\end{definition}
\begin{figure}
\begin{center}
\includegraphics[trim={0, 19.5cm 23.5cm 0},clip,scale=0.43]{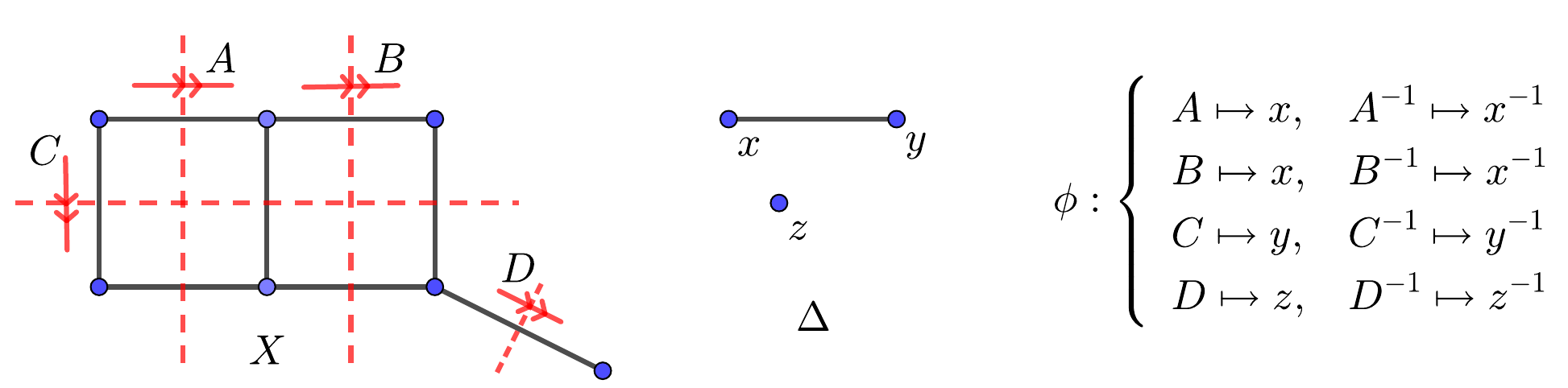}
\caption{Example of a special coloring.}
\label{SpecialColor}
\end{center}
\end{figure}

\noindent
See Figure \ref{SpecialColor} for an example. Essentially, a cube complex admits a special coloring if and only if it is special. We will not use this observation in this paper, so we only give sketch a proof at the end of this section. The main point to keep in mind is that any special cube complex admits at least one special coloring. 

\begin{lemma}\label{lem:specialcolorspecial}
A cube complex $X$ admits a special coloring if and only if there exists a special cube complex $Y$ such that $X^{(2)}=Y^{(2)}$.
\end{lemma}

\noindent
From now on, we fix a (non necessarily compact) special cube complex $X$ endowed with a special coloring $(\Delta,\phi)$. 

\medskip \noindent
Given a vertex $x_0 \in X$, a word $w=J_1 \cdots J_r$, where $J_1, \ldots, J_r \in V(\Delta) \sqcup V(\Delta)^{-1}$ are colors, is \emph{$x_0$-legal} if there exists a path $\gamma$ in $X$ starting from $x_0$ such that the oriented hyperplanes it crosses have colors $J_1, \ldots, J_r$ respectively. We say that the path $\gamma$ \emph{represents} the word $w$. 

\begin{fact}\label{fact:uniquepath}
An $x_0$-legal word is represented by a unique path in $X$. 
\end{fact}

\begin{proof}
Let $\gamma$ be a path representing a given $x_0$-legal word $w=J_1 \cdots J_r$. By definition, the starting vertex of $\gamma$ must be $x_0$. Also, we know that there exists at most one edge containing $x_0$ which has color $J_1$. Such an edge exists since $w$ is $x_0$-legal by assumption. A fortiori, this edge must be the first edge of $\gamma$. The conclusion follows by iterating the argument to the next edges.
\end{proof}

\noindent
The previous fact allows us to define the \emph{terminus} of an $x_0$-legal word $w$, denoted by $t(w)$, as the ending point of the unique path representing $w$. 

\medskip \noindent
Set $\mathcal{L}(X) = \{ x \text{-legal words} \mid x \in X\}$ the set of all legal words. (If $x_1,x_2 \in X$ are two distinct points, we consider the empty $x_1$-legal and $x_2$-legal words as distinct.) We consider the equivalence relation $\sim$ on $\mathcal{L}(X)$ generated by the following transformations:
\begin{description}
	\item[(cancellation)] if a legal word contains $JJ^{-1}$ or $J^{-1}J$, remove this subword;
	\item[(insertion)] given a color $J$, insert $JJ^{-1}$ or $J^{-1}J$ as a subword of a legal word;
	\item[(commutation)] if a legal word contains $J_1J_2$ where $J_1,J_2$ are two adjacent colors, replace this subword with $J_2 J_1$. 
\end{description}
So two $x$-legal words $w_1,w_2$ are equivalent with respect to $\sim$ if there exists a sequence of $x$-legal words 
$$m_1=w_1, \ m_2, \ldots, m_{r-1}, \ m_r=w_2$$ 
such that $m_{i+1}$ is obtained from $m_i$ by a cancellation, an insertion or a commutation for every $1 \leq i \leq r-1$. Define $\mathcal{D}(X)= \mathcal{L}(X)/ \sim$ as a set of \emph{diagrams}. The following observation allows us (in particular) to define the \emph{terminus} of a diagram as the terminus of one of the legal words representing it.

\begin{fact}\label{fact:transformations}
Let $w'$ be an $x_0$-legal word obtained from another $x_0$-legal word $w$ by a cancellation / an insertion / a commutation. If $\gamma',\gamma$ are paths representing $w',w$ respectively, then $\gamma'$ is obtained from $\gamma$ by removing a backtrack / adding a backtrack / flipping a square. 
\end{fact}

\begin{proof}
Suppose that $w'$ is obtained from $w$ by a cancellation. So we can write
$$w=J_1 \cdots J_r J J^{-1} J_{r+1} \cdots J_{r+s}$$ 
and $w'=J_1 \cdots J_{r+s}$ for some colors $J,J_1, \ldots, J_{r+s}$. Let $\gamma_1$ denote the subpath of $\gamma$ representing the $x_0$-legal word $J_1 \cdots J_r$. Let $e$ denote the oriented edge following $\gamma_1$ along $\gamma$. So the oriented hyperplane dual to $e$ has color $J$. After $e$, the oriented edge of $\gamma$ must be dual to an oriented hyperplane with color $J^{-1}$, as $e^{-1}$. But, because there exists at most one edge containing the terminus of $e$ which is dual to such a hyperplane, the only possibility is that $\gamma_1ee^{-1}$ is the subpath of $\gamma$ representing $J_1 \cdots J_r JJ^{-1}$. Now, if $\gamma_2$ denotes the subpath of $\gamma$ representing the $x_1$-legal word $J_{r+1} \cdots J_{r+s}$, where $x_1$ is the terminus of $\gamma_1$, then clearly $\gamma_1 \gamma_2$ represents $J_1 \cdots J_{r+s}$. It follows from Fact \ref{fact:uniquepath} that $\gamma'=\gamma_1 \gamma_2$. Consequently, $\gamma'$ is obtained from $\gamma$ by removing a backtrack. 

\medskip \noindent
By symmetry, we deduce that, if $w'$ is obtained from $w$ by an insertion, then $\gamma'$ is obtained from $\gamma$ by adding a backtrack.

\medskip \noindent
Now, suppose that $w'$ is obtained from $w$ by a commutation. So we can write
$$w=J_1 \cdots J_r H V J_{r+1} \cdots J_{r+s} \ \text{and} \ w'= J_1 \cdots J_r \cdots VH J_{r+1} \cdots J_{r+s},$$
for some colors $H,V,J_1, \ldots, J_{r+s}$ with $H$ and $V$ adjacent. Let $\gamma_1$ denote the subpath of $\gamma$ representing the $x_0$-legal word $J_1 \cdots J_r \subset w$. As a consequence of Fact \ref{fact:uniquepath}, this is also the subpath of $\gamma'$ representing the $x_0$-legal word $J_1 \cdots J_r \subset w'$. Let $x_1$ denote the terminus of $\gamma_1$. Because $w$ and $w'$ are both legal, we know that there exist two hyperplanes $A$ and $B$ which adjacent to $x_1$ and which have colors $H$ and $V$ respectively. Since $H$ and $V$ are adjacent colors, it follows that $x_1$ is the corner of a square $Q$ whose dual hyperplanes are precisely $A$ and $B$. If $x_2$ denotes the vertex of $Q$ opposite to $x_1$, let $\ell_1$ denote the path of length two in $Q$ from $x_1$ to $x_2$ passing through $A$ and then through $B$, and similarly $\ell_2$ the path of length two in $Q$ from $x_1$ to $x_2$ passing through $B$ and then through $A$. According to Fact \ref{fact:uniquepath}, $\ell_1$ (resp. $\ell_2$) is the unique $x_1$-legal word representing $HV$ (resp. $VH$), so $\gamma_1 \ell_1$ (resp. $\gamma_1 \ell_2$) must be the subpath of $\gamma$ (resp. $\gamma'$) representing $J_1 \cdots J_r HV$ (resp. $J_1 \cdots J_r VH$). Now, let $\gamma_2$ denote the subpath of $\gamma$ representing the $x_2$-legal word $J_{r+1} \cdots J_{r+s} \subset w$; this is also the subpath of $\gamma'$ representing the $x_2$-legal word $J_{r+1} \cdots J_{r+s} \subset w'$. Because $\gamma_1 \ell_1 \gamma_2$ represents $w$ by construction, it follows from Fact \ref{fact:uniquepath} that $\gamma = \gamma_1 \ell_1 \gamma_2$. Similarly, $\gamma'= \gamma_1 \ell_2 \gamma_2$. The conclusion follows since $\ell_2$ is obtained from $\ell_1$ by flipping the square $Q$. 
\end{proof}

\noindent
In the sequel, an \emph{$(x,y)$-diagram} will refer to a diagram represented by an $x$-legal word with terminus $y$, or just an \emph{$(x,\ast)$-diagram} if we do not want to mention its terminus. A diagram which is an $(x,x)$-diagram for some $x \in X$ is \emph{spherical}. 

\medskip \noindent
If $w$ is an $x_0$-legal word and $w'$ a $t(w)$-legal word, we define the \emph{concatenation} $w \cdot w'$ as the word $ww'$, which $x_0$-legal since it is represented by the concatenation $\gamma \gamma'$ where $\gamma, \gamma'$ represent respectively $w,w'$. Because we have the natural identifications
\begin{table}[h]
	\centering
		\begin{tabular}{cccc}
			$\mathcal{L}(X)$ & $\leftrightarrow$ & paths in $X$ & (Fact \ref{fact:uniquepath}) \\ 
			$\sim$ & $\leftrightarrow$ & homotopy with fixed endpoints & (Fact \ref{fact:transformations}, Proposition \ref{prop:cubehomotopy}) 
		\end{tabular}
\end{table}
\newpage
\noindent
it follows that the concatenation in $\mathcal{L}(X)$ induces a well-defined operation in $\mathcal{D}(X)$, making $\mathcal{D}(X)$ isomorphic to the fundamental groupoid of $X$. As a consequence, if we denote by $M(X)$ the Cayley graph of the groupoid $\mathcal{D}(X)$ with respect to the generating set $V(\Delta) \sqcup V(\Delta)^{-1}$, and, for every $x \in X$, $M(X,x)$ the connected component of $M(X)$ containing the trivial path $\epsilon(x)$ based at $x$, and $\mathcal{D}(X,x)$ the vertex-group of $\mathcal{D}(X)$ based at $\epsilon(x)$, then the previous identifications induce the identifications
\begin{table}[h]
	\centering
		\begin{tabular}{ccc}
			$\mathcal{D}(X)$ & $\leftrightarrow$ & fundamental groupoid of $X$  \\ 
			$\mathcal{D}(X,x)$ & $\leftrightarrow$ & $\pi_1(X,x)$  \\ 
			$M(X,x)$ & $\leftrightarrow$ & universal cover $\widetilde{X}^{(1)}$  \\ 
		\end{tabular}
\end{table}

\noindent
More explicitly, $\mathcal{D}(X,x)$ is the group of $(x,x)$-diagrams endowed with the concatenation, and $M(X,x)$ is the graph whose vertices are the $(x,\ast)$-diagrams and whose edges link two diagrams $w_1$ and $w_2$ if there exists some color $J$ such that $w_2=w_1J$. 

\medskip \noindent 
The identification between $M(X,x)$ and $\widetilde{X}^{(1)}$ will be fundamental in the sequel, so below is a more precise description of it. 
\begin{table}[h!]
	\centering
		\begin{tabular}{l}
			\hspace{5cm} $M(X,x) \longleftrightarrow \left( \widetilde{X}, \widetilde{x} \right)$ \\ \\ 
			\begin{tabular}{c} $(x,\ast)$-diagram represented \\ by an $x$-legal word $w$ \end{tabular}  $\mapsto$ \begin{tabular}{c} path $\gamma \subset X$ \\ representing $w$ \end{tabular}  $\mapsto$  \begin{tabular}{c} lift $\widetilde{\gamma} \subset \widetilde{X}$ of $\gamma$ \\ starting from $\widetilde{x}$ \end{tabular} $\mapsto$ \begin{tabular}{c} ending \\ point of $\widetilde{\gamma}$ \end{tabular} \\ \\
			\begin{tabular}{c} $(x,\ast)$-diagram represented by the \\ $x$-legal word corresponding to $\gamma$ \end{tabular} $\mapsfrom$ \begin{tabular}{c} image \\ $\gamma \subset X$ of $\gamma$ \end{tabular} $\mapsfrom$ \begin{tabular}{c} path $\widetilde{\gamma} \subset \widetilde{X}$ \\ from $\widetilde{x}$ to $y$ \end{tabular} $\mapsfrom$ $y$
		\end{tabular}
\end{table}

\noindent
It is worth noticing that, if we color the oriented edges of $X$ as their dual hyperplanes, then the (oriented) edges of $\widetilde{X}$ are naturally labelled by colors and vertices of $\Delta$, just by considering their images in $X$. A consequence of the previous paragraph is that the generator labelling a given oriented edge of $M(X,x)$ is the same as the color labelling the corresponding edge of $\widetilde{X}$. Below are two observations which we record for future use.

\begin{lemma}\label{lem:labeledge}
Two oriented edges of $\widetilde{X}$ dual to the same oriented hyperplane are labelled by the same color.
\end{lemma}

\begin{proof}
The images in $X$ of two such edges of $\widetilde{X}$ are dual to the same hyperplane as well, so that they must be labelled by the same color. The conclusion follows. 
\end{proof}

\begin{lemma}\label{lem:transverseadjacent}
If the hyperplanes dual to two given edges of $\widetilde{X}$ are transverse, then these edges are labelled by adjacent colors. 
\end{lemma}

\begin{proof}
The images in $X$ of these two hyperplanes of $\widetilde{X}$ must be transverse as well, so they are labelled by adjacent colors. The conclusion follows. 
\end{proof}

\noindent
A diagram may be represented by several legal words. Such a word is \emph{reduced} if it has minimal length, ie., it cannot be shortened by applying a sequence of cancellations, insertions and commutation. It is worth noticing that, in general, a diagram is not represented by a unique reduced legal word, but two such reduced words differ only by some commutations. (For instance, consider the homotopically trivial loop defined by the paths representing two of our reduced legal words, consider a disc diagram of minimal area bounded by this loop, and follow the proof of \cite[Theorem 4.6]{MR1347406}. Alternatively, use the embedding proved in the next section (which does not use the present discussion) and conclude by applying the analogous statement which holds in right-angled Artin groups.) As a consequence, we can define the \emph{length} of a diagram as the length of any reduced legal word representing it. It is worth noticing that our length coincides with the length which is associated to the generating set $V(\Delta) \sqcup V(\Delta)^{-1}$ in the groupoid $\mathcal{D}(X)$. The next lemma follows from this observation.

\begin{lemma}
Let $D_1,D_2 \in M(X,x)$ be two $(x,\ast)$-diagrams. If $J_1 \cdots J_n$ is a reduced legal word representing $D_1^{-1}D_2$, then 
$$D_1, \ D_1J_1, \ D_1J_1J_2, \ldots, \ D_1J_1 \cdots J_n$$
is a geodesic from $D_1$ to $D_2$ in $M(X,x)$. Conversely, any geodesic between $D_1$ and $D_2$ arises in this way.
\end{lemma}

\noindent
We end this section with a sketch of proof of Lemma \ref{lem:specialcolorspecial}, as promised. 

\begin{proof}[Sketch of proof of Lemma \ref{lem:specialcolorspecial}.]
Let $Y$ be a special cube complex. Since its hyperplanes are two-sided, we can orient them in order to identify $V ( \vec{\Delta} Y)$ with $V(\Delta Y) \sqcup V(\Delta Y)^{-1}$. We claim that $(\Delta Y, \mathrm{id})$ is a special coloring of $Y$. The first two points in the definition of a special coloring are clearly satisfied; the third point is satisfied since $Y$ does not contain self-intersecting and self-osculating hyperplanes; and the last one follows from the fact that $Y$ does not contain inter-osculating hyperplanes. Now, if $X^{(2)}=Y^{(2)}$, then the hyperplanes of $X$ can be naturally identified with those of $Y$, so that a special coloring of $Y$ induces a special coloring of $X$. 

\medskip \noindent
Conversely, suppose that $X$ admits a special coloring $(\Delta, \phi)$. From now on, we fix a basepoint $x \in X$ and we identify the one-skeleton of the universal cover $\widetilde{X}$ of $X$ with the graph $M(X,x)$. 

\medskip \noindent
We claim that $M(X,x)$ is a median graph. One possibility is to show that $M(X,x)$ is a triangle-free quasi-median graph by following the proof of \cite[Proposition 8.1]{Qm}. The conclusion follows from \cite[(25) p. 149]{Mulder}. One may notice that the median point of three $(x,\ast)$-diagrams $D_1,D_2,D_3$ is $D_1M$ where $M$ is the maximal common prefix of $D_1^{-1} D_2$ and $D_1^{-1} D_3$, ie., an $(t(D_1),\ast)$-diagram of maximal length appearing as a common prefix in some reduced legal words representing $D_1^{-1}D_2$ and $D_1^{-1}D_3$. 

\medskip \noindent
Therefore, it follows from \cite[Theorem 6.1]{mediangraphs} that filling in the cubes of $M(X,x)$ produces a CAT(0) cube complex $\widetilde{Y}$. Since $\widetilde{X}$ is obtained from (its one-skeleton) $M(X,x)$ by filling in some cubes of $M(X,x)$, necessarily $\widetilde{Y}$ is also obtained from $\widetilde{X}$ by filling in its cubes. As a consequence, the action $\pi_1(X) \curvearrowright \widetilde{X}$ extends to an action $\pi_1(X) \curvearrowright \widetilde{Y}$, so that $Y = \widetilde{Y} / \pi_1(X)$ is obtained from $X= \widetilde{X}/ \pi_1(X)$ by adding some cubes. On the other hand, since $\widetilde{X}$ is simply connected, any cycle of length four in its one-skeleton must bound a square, so that $\widetilde{Y}^{(2)}= \widetilde{X}^{(2)}$, and a fortiori $Y^{(2)}=X^{(2)}$. 

\medskip \noindent
So the remaining point to show is that $Y$ is special. Notice that one can naturally identify the hyperplanes of $Y$ with those of $X$ since $Y^{(2)}= X^{(2)}$, so that the special coloring of $X$ induces a special coloring of $Y$. The first point in the definition of a special coloring implies that the hyperplanes of $Y$ are two-sided (since the two orientations of a given hyperplanes have different colors); the third point implies that $Y$ does not contain self-intersecting and self-osculating hyperplanes; and the fourth point implies that it does not contain inter-osculating hyperplanes. Consequently, $Y$ is indeed a special cube complex. 
\end{proof}

\section{Warm up: a few applications}\label{section:warmup}

\subsection{Embedding in a right-angled Artin group}\label{section:embedRAAG}

\noindent
This section is dedicated to the proof of the following statement.  

\begin{thm}\label{thm:embedRAAG}
Let $X$ be a special cube complex and $x \in X$ a basepoint. If $(\phi,\Delta)$ is a special coloring of $X$, then the canonical map $\mathcal{D}(X,x) \to A(\Delta)$ induces an injective morphism. In particular, the fundamental group of $X$ embeds in the right-angled Artin group $A(\Delta)$. 
\end{thm}

\noindent
We begin by proving the following preliminary lemma.

\begin{lemma}\label{lem:commutationlegal}
If $w'$ is a word obtained from an $x$-legal word $w$ by a commutation or a cancellation, then $w'$ is again an $x$-legal word.
\end{lemma}

\begin{proof}
Suppose first that $w'$ is obtained from $w$ by a commutation. So we can write 
$$w=J_1 \cdots J_r HV J_{r+1} \cdots J_{r+s} \ \text{and} \ w'= J_1 \cdots J_r VH J_{r+1} \cdots J_{r+s}$$
for some colors $H,V, J_1, \ldots, J_{r+s}$ with $H$ and $V$ adjacent. The path $\gamma$ representing $w$ can be written as $\gamma_1 e_1e_2 \gamma_2$ where $\gamma_1$ is the subpath of $\gamma$ representing the $x$-legal word $J_1 \cdots J_r$, $e_1$ and $e_2$ the next two edges along $\gamma$, respectively dual to oriented hyperplanes $A$ and $B$ which have colors $H$ and $V$ respectively, and $\gamma_2$ the subpath of $\gamma$ representing the $t(J_1 \cdots J_r HV)$-legal word $J_{r+1} \cdots J_{r+s}$. Because $H$ and $V$ are adjacent colors, necessarily the edges $e_1$ and $e_2$ must generate a square $Q$ whose dual hyperplanes are precisely $A$ and $B$. If $e_1'e_2'$ is the path obtained from $e_1e_2$ by flipping $Q$, then $\gamma_1 e_1'e_2' \gamma_2$ is a path representing $w'$. Therefore, $w'$ is $x$-legal. 

\medskip \noindent
Now, suppose that $w'$ is obtained from $w$ by a cancellation. So we can write
$$w=J_1 \cdots J_r JJ^{-1} J_{r+1} \cdots J_{r+s} \ \text{and} \ w'=J_1 \cdots J_r J_{r+1} \cdots J_{r+s}$$
for some colors $J,J_1, \ldots, J_{r+s}$. Let $\gamma$ denote the path representing $w$, and let $\gamma_1$ denote the subpath of $\gamma$ representing the $x$-legal word $J_1 \cdots J_r$ and $e$ the next edge along $\gamma$. So $e$ is dual to $J$. The edge following $e$ along $\gamma$ must be dual to a hyperplane which has color $J^{-1}$; since we know that there may exist at most one such edge, it follows that this edge must be $e^{-1}$. Therefore, $\gamma ee^{-1}$ is the subpath of $\gamma$ representing the $x$-legal word $J_1 \cdots J_r JJ^{-1}$. Let $\gamma_2$ denote the subpath of $\gamma$ representing the $t(J_1 \cdots J_r)$-legal word $J_{r+1} \cdots J_{r+s}$, so that $\gamma = \gamma_1 ee^{-1} \gamma_2$. Then $\gamma_1\gamma_2$ is a path representing $w'$, so $w'$ is an $x$-legal word. 
\end{proof}

\begin{proof}[Proof of Theorem \ref{thm:embedRAAG}.]
Consider the canonical map $\mathcal{L}(X) \to A(\Delta)$. By noticing that this map is invariant under cancellations, insertions, and commutations, we get a well-defined morphism $\mathcal{D}(X,x) \to A(\Delta)$, where $x \in X$ is a vertex we fix. The injectivity of this morphism follows from the following fact, which we record for future use:

\begin{fact}\label{fact:diagvsRAAG}
Let $X$ be a special cube complex, $x \in X$ a basepoint and $(\phi,\Delta)$ a special coloring of $X$. Two $x$-legal words of colors are equal in $\mathcal{D}(X,x)$ if and only if they are equal in $A(\Delta)$. 
\end{fact}

\noindent
Indeed, it follows from the normal form in right-angled Artin groups (see for instance \cite[Theorem 3.9]{GreenGP}) that two $x$-legal words are equal in $A(\Delta)$ if and only if their reductions (obtained by applying commutations and cancellations) differ only by applying commutations. Similarly, we know that these words are equal in $\mathcal{D}(X,x)$ if and only if their reductions (obtained by applying commutations and cancellations) differ only by applying commutations, with the restriction that all these transformations must produce, at each step, another legal word. Therefore, our fact follows from the previous lemma.

\medskip \noindent
This shows that our morphism $\mathcal{D}(X,x) \to A(\Delta)$ is injective. 
\end{proof}

\begin{remark}
As a particular case of Theorem \ref{thm:embedRAAG}, it follows that the fundamental group of a special cube complex $X$ embeds in $A(\Delta X)$, where $\Delta X$ denotes the crossing graph of $X$. This statement was originally proved in \cite{MR2377497}. In fact, our theorem also follows from the same arguments. Indeed, if $S(\Delta)$ denotes the Salvetti complex associated to the graph $\Delta$, it is not difficult to deduce from the definition of a special coloring that the map sending an oriented edge $e$ of $X$ to the (unique) oriented edge of $S(\Delta)$ labelled by the color of $e$ induces a local isometric embedding $X \hookrightarrow S(\Delta)$. Since a local isometric embedding between nonpositively curved cube complexes turns out to be $\pi_1$-injective, the conclusion follows. Moreover, we get exactly the same injective morphism $\pi_1(X) \hookrightarrow A(\Delta)$. 
\end{remark}

\noindent
The interest of Theorem \ref{thm:embedRAAG} compared to the original statement in \cite{MR2377497} is that finding a good special coloring of a some special cube complex allows us to embed its fundamental group in a smaller right-angled Artin group. We give an example below. 

\begin{ex}\label{ex:embedRAAG}
Let $X$ denote the Squier cube complex associated to the semigroup presentation 
$$\left\langle a_1,a_2,a_3,b_1,b_2,b_3,p \left| \ a_1=a_1p, \ b_1=pb_1, \begin{array}{l} a_1=a_2, a_2=a_3,a_3=a_1 \\ b_1=b_2, b_2=b_3, b_3=b_1 \end{array} \right. \right\rangle$$
and the baseword $a_1b_1$. We refer to \cite{MR1396957, arXiv:1507.01667} for more details on Squier complexes. The sequel is largely based on \cite[Example 3.17]{arXiv:1507.01667}. The fundamental group of $X$ is
$$\mathbb{Z} \bullet \mathbb{Z} = \langle a,b,z \mid [a,z^nbz^{-n}]=1, n \geq 0 \rangle.$$
Moreover, $X$ is a special cube complex with eight hyperplanes, denoted by $A_i=[1, a_i \to a_{i+1},b_1]$, $B_i=[a_1,b_i \to b_{i+1},1]$, $C=[1,a_1,a_1p,b_1]$ and $D=[a_1,b_1 \to pb_1,1]$ where $i \in \mathbb{Z}_3$. The crossing graph of $X$ is a complete bipartite graph $K_{4,4}$, so that the corresponding right-angled Artin group is the product of free groups
$$\mathbb{F}_4 \times \mathbb{F}_4 = \langle A_1,A_2,A_3, C \mid \ \rangle \ast \langle B_1,B_2,B_3,D \mid \ \rangle.$$ 
One obtains the embedding
$$\left\{ \begin{array}{ccc} \mathbb{Z} \bullet \mathbb{Z} & \to & \mathbb{F}_4 \times \mathbb{F}_4 \\ a & \mapsto & A_1A_2A_3^{-1} \\ b & \mapsto & B_1B_2B_3^{-1} \\ z & \mapsto & CD^{-1} \end{array} \right.$$
Now, let $\Delta$ be a square with vertices $A,B,C,D$ (in a cyclic order), and consider the coloring $\phi$ sending the $A_i$'s to $A$, the $B_i$'s to $B$, the oriented hyperplane $C$ to the vertex $C$, and the oriented hyperplane $D$ to the vertex $D$. So the corresponding right-angled Artin group becomes the product of free groups
$$\mathbb{F}_2 \times \mathbb{F}_2 = \langle A,C \mid \ \rangle \times \langle B,D \mid \ \rangle,$$
and the embedding one gets is
$$\left\{ \begin{array}{ccc} \mathbb{Z} \bullet \mathbb{Z}  & \to & \mathbb{F}_2 \times \mathbb{F}_2 \\ a & \mapsto & A^3 \\ b & \mapsto & B^3 \\ z & \mapsto & CD^{-1} \end{array} \right.$$
Moreover, by noticing that the subgroup $\langle A^3,B^3, C,D^{-1} \rangle \leq \mathbb{F}_2 \times \mathbb{F}_2$ is naturally isomorphic to $\mathbb{F}_2 \times \mathbb{F}_2$ itself, if follows that the subgroup $\langle A,B,CD \rangle$ of $\mathbb{F}_2 \times \mathbb{F}_2$ is isomorphic to the infinitely presented group $\mathbb{Z} \bullet \mathbb{Z}$. 
\end{ex}

\subsection{Periodic flats}\label{section:flat}

\noindent
This section is dedicated to the proof of the following statement. 

\begin{thm}\label{thm:periodicflat}
Let $X$ be a compact special cube complex. If its universal cover contains an $n$-dimensional combinatorial flat for some $n \geq 1$, then $\pi_1(X)$ contains $\mathbb{Z}^n$.
\end{thm}

\begin{proof}
Suppose that, for some $n \geq 1$, there exists a combinatorial embedding $[0,+ \infty)^n \hookrightarrow \widetilde{X}$, where $\widetilde{X}$ denotes the universal cover of $X$. For convenience, we identify $[0,+ \infty)^n$ with its image in $\widetilde{X}$. From now on, we fix a basepoint $x \in X$ which admits $\widetilde{x}=(0, \ldots, 0)$ as a lift in $\widetilde{X}$, and we identify the one-skeleton of $\widetilde{X}$ with the graph $M(X,x)$. For every $1 \leq i \leq n$ and every $j \geq 0$, set
$$x^i(j) = (\underset{i-1 \ \text{zeros}}{\underbrace{0, \ldots, 0}}, j, 0, \ldots, 0),$$
and write $x^i(j)$ as a word $H^i_1 \cdots H^i_j$, where $H^i_1, \ldots, H^i_j$ are colors, which is represented by the image in $X$ of the unique geodesic segments from $\widetilde{x}$ to $x^i(j)$ in $[0,+ \infty)^n$. Let $A^i_1, \ldots, A^i_j$ denote the corresponding sequence of hyperplanes. 

\medskip \noindent
Let us begin by proving two preliminary facts.

\begin{fact}\label{fact:hyptransverse}
For all distinct $1 \leq i,j \leq n$ and all $p,q \geq 1$, the colors $H^i_p$ and $H^j_q$ are adjacent. 
\end{fact}

\noindent
By construction, $A^i_p$ (resp. $A^j_q$) has a lift $B^i_p$ (resp. $B^j_q$) in $\widetilde{X}$ separating $\widetilde{x}$ and $x^i(p)$ (resp. $\widetilde{x}$ and $x^j(q)$). Since the hyperplanes $B^i_p$ and $B^i_q$ are transverse to distinct factors of our combinatorial flat, they must be transverse, so that $A^i_p$ and $A^j_q$ must be transverse as well. Thus, it follows from Lemma \ref{lem:transverseadjacent} that $H^i_p$ and $H^j_q$ are adjacent colors, proving our fact.

\begin{fact}\label{fact:coordinates}
For all $p_1, \ldots, p_n \geq 1$, one has
$$(p_1, \ldots, p_n) = H^1_1 \cdots H^1_{p_1} \cdots H^n_1 \cdots H^n_{p_n}.$$
\end{fact} 

\noindent
Consider the following path from $(0, \ldots, 0)$ to $(p_1, \ldots, p_n)$:
$$\left( [0,p_1] \times \{ (0,\ldots, 0)\} \right) \cup \left( \{p_1 \} \times [0,p_2] \times \{(0, \ldots, 0)\} \right) \cup \cdots \cup \left( \{p_1, \ldots, p_{n-1}\} \times [0,p_n] \right).$$
The word obtained by reading the colors labelling the edges of this path defines a diagram representing $(p_1, \ldots, p_n)$. By noticing that the edge
$$\{(p_1, \ldots, p_k) \} \times [i,i+1] \times \{ (p_{k+2}, \ldots, p_n) \}$$
is dual to the same hyperplane as $\{(0,\ldots, 0)\} \times [i,i+1] \times \{ (0, \ldots, 0)\}$, it follows from Lemma \ref{lem:labeledge} that they are both labelled by $H^{k+1}_{i+1}$. Our fact follows. 

\medskip \noindent
Now, since $X$ is compact, for every $1 \leq i \leq n$, there exist $p_i < q_i$ such that $x^i(p_i)$ and $x^i(q_i)$ have the same terminus, say $t_i \in X$. Set $o=(p_1, \ldots, p_n)$, and
$$a_i = (p_1, \ldots, p_{i-1}, q_i, p_{i+1}, \ldots, p_n)$$
for every $1 \leq i \leq n$. It follows from Fact \ref{fact:coordinates} that, for every $1 \leq i \leq n$, 
$$o^{-1}a_i = x^i(p_i)^{-1}x^i(q_i) = H^i_{p_{i+1}} \cdots H^i_{q_i};$$
moreover, since $x^i(p_i)$ and $x^i(q_i)$ have the same terminus, it follows that $o^{-1}a_i$ is spherical. Consequently, the diagrams $o^{-1}a_1, \ldots, o^{-1}a_n$ are non-trivial elements of $\mathcal{D}(X,t)$, where $t$ denotes the terminus of $o$. Thus, we have proved the following statement:

\begin{fact}\label{fact:flatfreeabeliansub}
Let $X$ be a special coloring, $x \in X$ a basepoint and $(\Delta, \phi)$ a special coloring. If the universal cover of $X$ contains an $n$-dimensional combinatorial flat, then $\mathcal{D}(X,x)$ contains $n$ non-trivial diagrams $g_1, \ldots, g_n$ such that, as words of colors, any color of $g_i$ is adjacent to any color of $g_j$ for all distinct $1 \leq i,j \leq n$. 
\end{fact}

\noindent
Clearly, $\langle g_1, \ldots, g_n \rangle$ defines a subgroup of $\mathcal{D}(X,x) \simeq \pi_1(X)$ which is isomorphic to $\mathbb{Z}^n$, proving our theorem.
\end{proof}

\begin{cor}
A virtually cocompact special group is hyperbolic if and only if it does not contain $\mathbb{Z}^2$. 
\end{cor}

\begin{proof}
Let $G$ be a group containing a finite-index subgroup $\dot{G}$ which is the fundamental group of a compact special cube complex $X$. One already knows that, if $G$ is hyperbolic, then it does not contain $\mathbb{Z}^2$. Conversely, if $G$ is not hyperbolic, then $\dot{G}$ cannot be hyperbolic, so that the universal cover $\widetilde{X}$ must contain a two-dimensional combinatorial flat (see for instance \cite[Corollary 5]{CDEHV} or \cite[Theorem 3.3]{coningoff}).

\medskip \noindent
As a by-product, notice that the previous observation, together with Fact \ref{fact:flatfreeabeliansub}, produces the following statement, which we record for future use:

\begin{fact}\label{fact:nothyp}
Let $X$ be a cube complex and $(\Delta, \phi)$ a special coloring. If $\pi_1(X)$ is not hyperbolic, then there exist two commuting spherical diagrams $J_1 \cdots J_n$ and $H_1 \cdots H_m$ such that $J_i$ and $H_j$ are adjacent colors for every $1 \leq i \leq n$ and every $1 \leq j \leq m$. 
\end{fact}

\noindent
Anyway, it follows from Theorem \ref{thm:periodicflat} that $\dot{G}$, and a fortiori $G$, contains $\mathbb{Z}^2$, proving our corollary. 
\end{proof}

\section{Malnormal subgroups and relative hyperbolicity}\label{section:RHspecial}

\subsection{General relative hyperbolicity}\label{section:RH}

\noindent
In this section, we are interested in relatively hyperbolic special groups. Our main general criterion is the following:

\begin{thm}\label{thm:specialrelhyp}
Let $G$ be a special group and $\mathcal{H}$ a finite collection of subgroups. Then $G$ is hyperbolic relative to $\mathcal{H}$ if and only if the following conditions are satisfied:
\begin{itemize}
	\item each subgroup of $\mathcal{H}$ is convex-cocompact;
	\item $\mathcal{H}$ is an almost malnormal collection;
	\item every non virtually cyclic abelian subgroup of $G$ is contained in a conjugate of some group of $\mathcal{H}$. 
\end{itemize}
\end{thm}

\noindent
The statement is an algebraic version of a geometric intuition which can be found in other criteria of relative hyperbolicity, such as \cite{isolated, MR2469527}. 

\medskip \noindent
First of all, let us mention the definition of relative hyperbolicity which we use. (The definition comes from \cite{relativelyhyperbolic}, and we refer to \cite{MR2684983} for proofs of its equivalence with other definitions.)

\begin{definition}
A finitely generated group $G$ is \emph{hyperbolic relative to a collection of subgroups $\mathcal{H}=\{ H_1, \ldots, H_n \}$} if $G$ acts by isometries on a graph $\Gamma$ such that:
\begin{itemize}
	\item $\Gamma$ is hyperbolic;
	\item $\Gamma$ contains finitely-many orbits of edges;
	\item edge-stabilisers are finite;
	\item each vertex-stabilizer is either finite or is a conjugate of some $H_i$,
	\item any $H_i$ stabilises a vertex,
	\item $\Gamma$ is \emph{fine}, ie., any edge belongs only to finitely-many simple loops (or \emph{cycle}) of a given length.
\end{itemize}
A subgroup conjugate to some $H_i$ is \emph{peripheral}. 
\end{definition}

\begin{definition}
A finitely generated group is \emph{relatively hyperbolic} if it is hyperbolic relative to a finite collection of proper subgroups.
\end{definition}

\noindent
The first step towards the proof of Theorem \ref{thm:specialrelhyp} is to understand what the malnormality of our collection of subgroups implies geometrically. 

\begin{prop}\label{prop:malnormalcontracting}
Let $X$ be a compact special cube complex, $(\Delta, \phi)$ a special coloring and $x \in X$ a basepoint. If $H \subset \mathcal{D}(X,x)$ is a malnormal convex-cocompact subgroup, then the convex hull of $H$ in $M(X,x)$ is contracting.
\end{prop}

\noindent
We begin by proving two preliminary lemmas.

\begin{lemma}\label{lem:convexhullgeom}
Let $\widetilde{X}$ be a finite-dimensional CAT(0) cube complex and $G \leq \mathrm{Isom}(X)$ a convex-cocompact subgroup. For every $x \in \widetilde{X}$, $G$ acts geometrically on the convex hull of the orbit $G \cdot x$.
\end{lemma}

\begin{proof}
Let $Y$ be a convex subcomplex of $\widetilde{X}$ on which $G$ acts geometrically. Let $Z$ denote a neighborhood of $Y$ with respect to the $\ell^{\infty}$-metric, ie., the metric obtained by extending the $\ell^{\infty}$-norms defined on each cube of $\widetilde{X}$, which contains the basepoint $x$. According to \cite[Corollary 3.5]{packing}, $Z$ is a convex subcomplex. Moreover, since $\widetilde{X}$ is finite-dimensional, our new metric is quasi-isometric to the old one. Necessarily $Z$ lies in a neighborhood of $Y$, so that $G$ acts geometrically on $Z$ as well. Therefore, there exists a constant $K \geq 0$ such that $Z$ is covered by $G$-translates of the ball $B(x,K)$. Clearly, the convex hull $C$ of the orbit $G \cdot x$ is contained in $Z$, so $C$ is also covered by $G$-translates of the ball $B(x,K)$, showing that $G$ acts geometrically on~$C$. 
\end{proof}

\noindent
In our second lemma, and in the rest of the section, we use the following convenient notation: if $X$ is a metric space and $S \subset X$ a subset, for every $L \geq 0$ we denote by $S^{+L}$ the $L$-neighborhood of $S$, namely $\{ x \in X \mid d(x,S) \leq L\}$. 

\begin{lemma}\label{lem:malnormalS}
Let $G$ be a group acting metrically properly on a geodesic metric space $S$, $\{ H_1, \ldots, H_m \}$ a malnormal collection of subgroups, and $x \in S$ a basepoint. Suppose that, for every $1 \leq i \leq m$, there exists a subspace $S_i \supset H_i \cdot x$ on which $H_i$ acts geometrically. For every $L \geq 0$, there exists some $K \geq 0$ such that 
$$\mathrm{diam} \left( S_i^{+L} \cap g S_j^{+L} \right) \leq K$$
whenever $i \neq j$ and $g \in G$ or $i=j$ and $g \in G \backslash \{ H_i \}$. 
\end{lemma}

\begin{proof}
Let $C \geq 0$ a constant such that, for every $1 \leq i \leq m$, $S_i$ is covered by $H_i$-translates of the ball $B(x,C)$. Fix some $L \geq 0$, some $1 \leq r,t \leq m$ and some $g \in G$, and suppose that the diameter of $S_r^{+L} \cap gS_t^{+L}$ is at least $n(2C+1)$ for some $n \geq 1$. As a consequence, there exist $a_1, \ldots , a_n \in S_r^{+L} \cap gS_t^{+L}$ such that $d(a_i,a_j) \geq 2C+1$ for all distinct $1 \leq i,j \leq n$. For every $1 \leq i \leq n$, fix $b_i \in S_r$ and $c_i \in gS_t$ such that $d(a_i,b_i) \leq L$ and $d(a_i,c_i) \leq L$. For every $1 \leq i \leq n$, there exist $h_i \in H_r$ and $h_k \in H_t^g$ such that $d(b_i,h_ix) \leq C$ and $d(c_i,k_ix) \leq C$. Notice that, for every $1 \leq i \leq n$, one has
$$d(h_ix,k_ix) \leq d(h_ix,b_i)+d(b_i,a_i)+d(a_i,c_i)+d(c_i,k_ix) \leq 2(L+C),$$
or equivalently, $d(x,h_i^{-1}k_i x) \leq 2(L+C)$. Now, because $G$ acts metrically properly on $S$, there exists some $N \geq 0$ such that at most $N$ elements of $G$ may satisfy this inequality. Consequently, if $n > N$, then $\{ h_i^{-1}k_i \mid 1 \leq i \leq n\}$ must contain at least two equal elements, say $h_1^{-1}k_1$ and $h_2^{-1}k_2$; equivalently, $h_1h_2^{-1}k_2k_1^{-1}=1$. For convenience, set $p=h_1h_2^{-1}=k_1k_2^{-1}$; notice that $p$ belongs to $S_r \cap S_t^g$. Next, one has
$$\begin{array}{lcl} d(pa_2, a_2) & \geq & d(a_2,a_1) -d(pa_2, a_1) \geq d(a_2,a_1)-d(h_2^{-1}a_2,x)- d(x,h_1^{-1}a_1) \\ \\ & \geq & 2C+1-C-C=1 \end{array}$$
A fortiori, $p \neq 1$, so that $H_i \cap H_j^g \neq \{1 \}$. It implies that $i=j$ and $g \in H_i$. We conclude that $K=N (2C+1)$ is the constant we are looking for. 
\end{proof}

\begin{proof}[Proof of Proposition \ref{prop:malnormalcontracting}.]
Let $H \subset \mathcal{D}(X,x)$ be a subgroup and let $Y$ denote the convex hull of $H$ in $M(X,x)$. For convenience, set $L=\# V(X)+1$. Suppose that there exists a flat rectangle $R : [0,n] \times [0,L] \hookrightarrow M(X,x)$ satisfying $Y \cap R = [0,n] \times \{0\}$ for some $n > \# V(\Gamma)$. Because the terminus of a diagram may take at most $\#V(X)$ different values, there must exist $0 \leq a < b \leq L$ such that $(0,a)$ and $(0,b)$ have the same terminus. Let $W$ be the diagram $(0,0)$, $A$ the diagram labelling the geodesic between $(0,0)$ and $(0,a)$ in $R$, and $B$ the diagram labelling the geodesic between $(0,a)$ and $(0,b)$. Notice that $B$ is a $(t(WA),t(WA))$-diagram, so that $g=WABA^{-1} W^{-1}$ belongs to $\mathcal{D}(X,x)$. Now, for every $1 \leq k \leq n$, let $C_k$ denote the diagram labelling the geodesic between $(0,0)$ and $(k,0)$ in $R$. Notice that, because any hyperplane intersecting $[0,n] \times \{ 0 \}$ is transverse to any hyperplane intersecting $\{0 \} \times [0,L]$, any color of the word $C_k$ is adjacent to any color of the words $A$ and $B$. Consequently, 
$$g \cdot WC_k = WABA^{-1}W^{-1} \cdot WC_k = W \cdot ABA^{-1} \cdot C_k = W \cdot C_k \cdot ABA^{-1},$$
so that $d(WC_k,gWC_k) = d(\epsilon(x),ABA^{-1}) \leq \mathrm{length}(ABA^{-1}) \leq 3L$. It follows that
$$\mathrm{diam} \left( Y^{+3L} \cap gY^{+3L} \right) \geq n.$$
But, since $H$ acts geometrically on $Y$ as a consequence of Lemma \ref{lem:convexhullgeom}, Lemma \ref{lem:malnormalS} applies, and it implies that $n$ cannot be arbitrarily large. The conclusion follows from Theorem \ref{thm:contractingCube}
\end{proof}

\noindent
Our next lemma will be useful when combined with Lemma \ref{lem:malnormalS}.

\begin{lemma}\label{lem:fellowtravel}
Let $X$ be a CAT(0) cube complex and $A,B \subset X$ two $L$-contracting convex subcomplexes. Suppose that any vertex of $X$ has at most $R \geq 2$ neighbors. If there exist $N \geq \max(L,2)$ hyperplanes transverse to both $A$ and $B$, then the inequality
$$\mathrm{diam} \left( A^{+L} \cap B^{+L} \right) \geq \ln(N-1)/\ln(R)$$
holds. 
\end{lemma}

\begin{proof}
Let $F$ denote the projection of $A$ onto $B$. Then $F$ is a convex subcomplex whose hyperplanes are precisely the hyperplanes transverse to both $A$ and $B$ (see for instance \cite[Proposition 2.9]{coningoff}). Let $x \in F$ be a vertex. By definition, there exists some vertex $y \in A$ such that $x$ is the projection of $y$ onto $B$. Let $y'$ denote the projection of $x$ onto $A$ and $y''$ the projection of $y'$ onto $B$. Suppose that there exists some hyperplane $J$ separating $x$ and $y''$. Because any hyperplane separating $y'$ and $y''$ must be disjoint from $B$, necessarily $J$ separates $x$ and $y'$. But, since any hyperplane separating $x$ and $y'$ must be disjoint from $A$, it follows that $J$ separates $x$ and $y$. On the other hand, a hyperplane separating $y$ and $x$ must be disjoint from $B$, hence a contradiction. Thus, we have prove that no hyperplane separates $x$ and $y''$, hence $x=y''$. Therefore, $y'$ is the projection of $x$ onto $A$ and $x$ is the projection of $y'$ onto $B$. It follows that the hyperplanes separating $x$ and $y'$ are precisely the hyperplanes separating $A$ and $B$. By looking at the collections of the hyperplanes separating $A$ and $B$, and the hyperplanes intersecting both $A$ and $B$, we conclude that $d(x,A) =d(x,y') \leq L$. Therefore, $F \subset A^{+L} \cap B$. On the other hand, we know that $F$ contains at least $N$ hyperplanes, so our lemma is implied by the following claim. 

\begin{claim}
Let $F$ be a finite CAT(0) cube complex. Suppose that any vertex of $F$ has at most $R \geq 2$ neighbors and that $F$ contains at least $N \geq 2$ hyperplanes. Then $\mathrm{diam}(F) \geq \ln(N-1)/\ln(R)$.
\end{claim}

\noindent
Let $d$ denote the diameter of $F$. From our assumptions, it is clear that $F$ contains at most $R^d+1$ vertices. On the other hand, the number of hyperplanes of $F$ is bounded above by its number of edges, and a fortiori by its number of vertices. Since $F$ contains at least $N$ hyperplanes, it follows that $N \leq R^d+1$, and the inequality we are looking for follows. 
\end{proof}

\noindent
The next statement is our last preliminary lemma before proving Theorem \ref{thm:specialrelhyp}. 

\begin{lemma}\label{lem:smallflats}
Let $\widetilde{X}$ be a locally finite CAT(0) cube complex and $\mathcal{C}$ a collection of convex subcomplexes which are all $P$-contracting for some $P \geq 0$. Suppose that there exists a constant $K \geq 0$ such that the image of every combinatorial isometric embedding $\mathbb{R}^2 \hookrightarrow \widetilde{X}$ lies in the $K$-neighborhood of some subcomplex $C \in \mathcal{C}$. There exist some constants $A,B \geq 0$ such that every $A$-thick flat rectangle of $\widetilde{X}$ lies in the $B$-neighborhood of some subcomplex of $\mathcal{C}$.
\end{lemma}

\begin{proof}
Our lemma will follow from the next claim.

\begin{claim}
Let $K \geq 0$ and $L > 2P $ be two constants, $R : [0,p] \times [0,q] \hookrightarrow \widetilde{X}$ an $L$-thick flat rectangle, and $C \in \mathcal{C}$ a subcomplex. Either $R$ lies in $C^{+K+2P}$, or $R \cap C^{+K}$ has diameter at most $2(P+K)$. 
\end{claim}

\noindent
Suppose that $\mathrm{diam} \left( R \cap C^{+K} \right) > 2(P+K)$. Fix two vertices $x,y \in R \cap C^{+K}$ satisfying $d(x,y) > 2(P+K)$, and let $x',y' \in C$ denote respectively the projections of $x,y$ onto $C$. Notice that a hyperplane separating $x$ and $y$ separates $x$ and $x'$, or $y$ and $y'$, or $x'$ and $y'$, so that there must exist more than $2P$ hyperplanes separating $x$ and $y$ and intersecting $C$. For convenience, let $\mathcal{V}$ (resp. $\mathcal{H}$) denote the hyperplanes intersecting $[0,p] \times \{0\}$ (resp. $\{0\} \times [0,q]$). Our previous observation implies that more than $P$ hyperplanes of $\mathcal{V}$ or of $\mathcal{H}$ intersect $C$; up to switching $\mathcal{H}$ and $\mathcal{V}$, say that we are in the former case. Since $C$ is $P$-contracting, at most $P$ hyperplanes of $\mathcal{H}$ may be disjoint from $C$. Next, because $R$ is $L$-thick with $L >2P $, we deduce that $\mathcal{H}$ contains more than $P$ hyperplanes intersecting $C$, so that, once again because $C$ is $P$-contracting, at most $P$ hyperplanes of $\mathcal{V}$ may be disjoint from $C$. Now, let $z \in R$ be a vertex. Notice that a hyperplane separating $z$ from $C$ separates either $x$ and $x'$ or $x$ and $z$. But there exist at most $K$ hyperplanes separating $x$ and $x'$, and at most $2P$ hyperplanes disjoint from $C$ separating $x$ and $z$ (these hyperplanes belonging necessarily to $\mathcal{H} \cup \mathcal{V}$). Therefore, $d(z,C) \leq K+2P$. Thus, we have proved that $R \subset C^{+K+2P}$, concluding the proof of our claim.

\medskip \noindent
Now we are ready to prove our lemma. Suppose that, for $K,L \geq 0$, there exists an $L$-thick flat rectangle which is not included in the $K$-neighborhood of any subcomplex of $\mathcal{C}$. As a consequence of our previous claim, we know that, for every $K \geq 0$ and every $L>2P$, there exists an $L$-thick flat rectangle $R$ satisfying $\mathrm{diam} \left( R \cap C^{+K} \right) \leq 2(P+K)$ for every $C \in \mathcal{C}$. So, fixing some $K \geq 0$, for every $n>2P$, there exists an $n$-thick flat rectangle $R_n$ satisfying $\mathrm{diam} \left( R_n \cap C^{+K} \right) \leq 2(P+K)$ for every $C \in \mathcal{C}$. Since $\widetilde{X}$ is locally finite, up to taking a subsequence, we may suppose without loss of generality that $(R_n)$ converges to some subcomplex $R_{\infty} \subset \widetilde{X}$, ie., for every ball $B$ the sequence $(R_n \cap B)$ is eventually constant to $R_{\infty} \cap B$. Notice that $R_{\infty}$ is isometric to the square complex $\mathbb{R}^2$ and that $\mathrm{diam} \left( R_{\infty} \cap C^{+K} \right) \leq 2(P+K)$ for every $C \in \mathcal{C}$. Thus, we have proved that, for every $K \geq 0$, there exists a combinatorial isometric embedding $\mathbb{R}^2 \hookrightarrow \widetilde{X}$ whose image is not contained in the $K$-neighborhood of any subcomplex $C \in \mathcal{C}$. This concludes the proof of our lemma. 
\end{proof}

\begin{proof}[Proof of Theorem \ref{thm:specialrelhyp}.]
Suppose that $G$ is hyperbolic relative to $\mathcal{H}$. Then the subgroups of $\mathcal{H}$ are convex-cocompact according to \cite{SageevWiseCores}; the collection $\mathcal{H}$ is almost malnormal according to \cite[Theorem 1.4]{OsinRelativeHyp}; and every non cyclic abelian subgroup of $G$ is contained in a conjugate of some group of $\mathcal{H}$ according to \cite[Theorem 4.19]{OsinRelativeHyp}. 

\medskip \noindent
Conversely, suppose that our three conditions are satisfied. Let $X$ be a compact special cube complex whose fundamental group is $G$. Fix a basepoint $x \in X$, and identify $G$ and the (one-skeleton of the) universal cover of $X$ with $\mathcal{D}(X,x)$ and $M(X,x)$ respectively. Because $\mathcal{D}(X,x)$ is naturally embedded in $M(X,x)$, one may naturally identify subsets of $\mathcal{D}(X,x)$ with subsets of $M(X,x)$. Let $\mathcal{C}$ denote the set of the convex hulls in $M(X,x)$ of the cosets of the subgroups of $\mathcal{H}$. 

\begin{claim}\label{claim:fine}
There exists a constant $K \geq 0$ such that, for all distinct $C_1,C_2 \in \mathcal{C}$, at most $K$ hyperplanes intersect both $C_1$ and $C_2$. 
\end{claim}

\noindent
Our claim is a consequence of Lemma \ref{lem:fellowtravel}, which applies thanks to Proposition \ref{prop:malnormalcontracting}, combined with Lemma \ref{lem:malnormalS}. 

\begin{claim}
An edge of $M(X,x)$ belongs to only finitely many subcomplexes of $\mathcal{C}$.
\end{claim}

\noindent
Suppose by contradiction that there exist infinitely many subcomplexes of $\mathcal{C}$ containing a given edge $e$. Because $M(X,x)$ is locally finite, this implies that, for every $n \geq 1$, there exist distinct subcomplexes $C_1,C_2 \in \mathcal{C}$ satisfying $C_1 \cap B(n) = C_2 \cap B(n)$, where $B(n)$ is a ball of radius $n$ centered at an endpoint of $e$. But such a phenomenon would contradict our previous claim. 

\begin{claim}\label{claim:flatneighborhood}
There exists some constant $K \geq 0$ such that the image of every combinatorial isometric embedding $\mathbb{R}^2 \hookrightarrow M(X,x)$ lies in the $K$-neighborhood of some subcomplex of $\mathcal{C}$. 
\end{claim}

\noindent
Let $\mathbb{R}^2 \hookrightarrow M(X,x)$ be a combinatorial isometric embedding. For convenience, we identify $\mathbb{R}^2$ with its image in $M(X,x)$. First of all, we consider the orthant $[0,+ \infty)^2$. Because the terminus of a diagram may take only finitely many values, the ray $[0,+ \infty) \times \{0\}$ must contain infinitely many vertices $u_1,u_2, \ldots$ having the same terminus; similarly, the ray $\{0 \} \times [0,+ \infty)$ must contain infinitely many vertices $v_1,v_2, \ldots$ having the same terminus. For every $i \geq 1$, let $A_i$ denote the diagram labelling the unique geodesic in $\mathbb{R}^2$ between $u_i$ and $u_{i+1}$ and $B_i$ the diagram labelling the unique geodesic in $\mathbb{R}^2$ between $v_i$ and $v_{i+1}$; notice that $A_i$ and $B_i$ are spherical diagrams. We also denote by $W$ the diagram corresponding to the vertex of $\mathbb{R}^n$ whose projections onto the axes are the $u_1$ and $v_1$. Because any hyperplane intersecting $[0,+ \infty) \times \{0\}$ is transverse to any hyperplane intersecting $\{ 0 \} \times [0,+ \infty)$, it follows that, for all $i,j \geq 1$, any color of the word $A_i$ is adjacent to any color of the word $B_j$. As a consequence, the subgroup 
$$\mathscr{W}= W \cdot \langle A_1,B_1,A_2,B_2, \ldots \rangle \cdot W^{-1}$$
naturally decomposes as $W \langle A_1, \ldots \rangle W^{-1} \times W \langle B_1, \ldots \rangle W^{-1}$. It follows from the next observation that $\mathscr{W}$ is included in a conjugate of some subgroup of $\mathcal{H}$.

\begin{fact}
If a subgroup of $G$ decomposes as a direct product of two infinite subgroups, then it must be included in a conjugate of some subgroup of $\mathcal{H}$.
\end{fact}

\noindent
Let $K_1 \times K_2 \subset G$ be such a subgroup. For all non-trivial elements $k_1 \in K_1$ and $k_2 \in K_2$, there exists a conjugate $C(k_1,k_2)$ of some subgroup of $\mathcal{H}$ containing $\langle k_1,k_2 \rangle$ since this subgroup is a free abelian group of rank two. Notice that, for all non-trivial elements $k_1,k_2 \in K_1$ and $k_3 \in K_2$, the intersection $C(k_1,k_3) \cap C(k_2,k_3)$ contains $\langle k_3 \rangle$, hence $C(k_1,k_3)=C(k_2,k_3)$ since the collection $\mathcal{H}$ is malnormal. A similar statement holds if we switch the roles of $K_1$ and $K_2$. Therefore, for all non-trivial elements $k_1,k_2 \in K_1$ and $k_3,k_4 \in K_2$, one has
$$C(k_1,k_3)=C(k_2,k_3)=C(k_2,k_4).$$
Let $H$ denote this common conjugate. Now, for all non-trivial elements $k_1 \in K_1$ and $k_2 \in K_2$, one has $k_1=k_1^2k_2 \cdot k_1^{-1}k_2^{-1} \in H \cdot H = H$ and $k_2 = k_1k_2^2 \cdot k_1^{-1}k_2^{-1} \in H \cdot H =H$. Therefore, $K_1$ and $K_2$ are included in $H$, hence $K_1 \times K_2 \subset H$. This concludes the proof of our fact.  

\medskip \noindent
So there exist some $H \in \mathcal{H}$ and $g \in G$ such that $\mathscr{W} \subset H^g$. If $C$ denotes the convex hull of the coset $gH$, so that $C$ belongs to $\mathcal{C}$, notice that
$$\mathscr{W} \cdot W \subset H^g \cdot W \subset \left( H^g \cdot g \right)^{+d(g,W)} = C^{+d(g,W)}.$$
Let $Y$ denote the $d(g,W)$-neighborhood of $C$ with respect to the $\ell^{\infty}$-metric, ie., the metric obtained by extending the $\ell^{\infty}$-norms defined on each cube of $\widetilde{X}$. According to \cite[Corollary 3.5]{packing}, $Y$ is a convex subcomplex. Moreover, because $\widetilde{X}$ is finite-dimensional, our new metric is quasi-isometric to the old one, so that there exists a constant $L \geq 0$ such that $Y \subset C^{+L}$. Finally, since $[u_1,+ \infty) \times [v_1,+ \infty)$ clearly lies in the convex hull of $\mathscr{W} \cdot W$, we deduce that $[u_1,+ \infty) \times [v_1,+ \infty) \subset Y \subset C^{+L}$. A fortiori, $[0,+ \infty) \times [0,+ \infty)$ lies in a neighborhood of $C$. Similarly, one finds some $C' \in \mathcal{C}$ such that $[0,+ \infty) \times [0,- \infty)$ lies in a neighbor of $C'$. It implies that there exists some $K \geq 0$ such that $\mathrm{diam} \left( C^{+K} \cap (C')^{+K} \right)$ is infinite, hence $C=C'$ as a consequence of Lemma \ref{lem:malnormalS}. Arguing similarly with $[0,- \infty)^2$ and next with $[0,-\infty) \times [0,+ \infty)$, we deduce that $\mathbb{R}^2$ lies in a neighborhood of $C$. 

\medskip \noindent
Now, we need to show that the constant defining this neighborhood can be chosen uniformly for every flat subcomplex. Fix some $n \geq 1$ which is sufficiently large so that the distance between the projections of $(0,0)$ and $(0,n)$ (resp. $(0,0)$ and $(n,0)$) is at least $L$, where $L$ is the constant provided by Lemma \ref{lem:geodcontracting}. (Notice that $L$ can be chosen independently of $C$ since $\mathcal{C}$ contains only finitely many $G$-orbits of subcomplexes.) Therefore, if $k$ denotes the distance from $(0,0)$ to $C$, then $(0,k)$ and $(k,0)$ must be at distance at most $L$ from the projection $p$ of $(0,0)$ onto $C$. Therefore, 
$$2k = d((k,0),(0,k)) \leq d((k,0),p)+d(p,(0,k)) \leq 2L$$
hence $d((0,0),C) = k \leq L$. Consequently, 
$$\mathscr{W} \cdot (0,0) \subset \mathscr{W} \cdot C^{+L} \subset H^g \cdot C^{+L} \subset C^{+L}.$$
Once again according to \cite[Corollary 3.5]{packing}, the $L$-neighborhood $Z$ of $C$ with respect to the $\ell^{\infty}$-metric is a convex subcomplex, so that the convex hull of $\mathscr{W} \cdot (0,0)$ must be included in $Z$ since $C^{+L}$ is contained in $Z$. On the other hand, because $\widetilde{X}$ is finite-dimensional, this new metric is quasi-isometric to the old one, so that there exists a constant $K$ depending only on $L$ such that $Y \subset C^{+K}$. We conclude that $\mathbb{R}^2$, which is clearly included in the convex hull of $\mathscr{W} \cdot (0,0)$, lies in the $K$-neighborhood of $C$, where $K$ does not depend on the flat we consider nor on $C$. This concludes the proof of our claim. 

\medskip \noindent
Let $\mathscr{C}$ be the graph obtained from $M(X,x)$ by adding a vertex $v_C$ for every $C \in \mathcal{C}$ and an edge between $v_C$ and every vertex of $C$. By combining Claim \ref{claim:flatneighborhood} and Lemma \ref{lem:smallflats} (which applies thanks to Proposition \ref{prop:malnormalcontracting}), we deduce from \cite[Theorem 4.1]{coningoff} that $\mathscr{C}$ is a hyperbolic graph. Moreover, according to Claim \ref{claim:fine}, it follows from \cite[Theorem 5.7]{coningoff} that $\mathscr{C}$ is fine. Notice that, because $G$ acts geometrically on $M(X,x)$ and that every subcomplex $C \in \mathcal{C}$ has a cocompact stabiliser, our graph $\mathscr{C}$ contains only finitely many $G$-orbits of edges. This proves that $G$ is hyperbolic relative to $\mathcal{H}$. 
\end{proof}

\noindent
As a remark, it is worth noticing that the characterisation of relatively hyperbolic right-angled Coxeter groups (proved in \cite{BHSC, coningoff}), and more generally the characterisation of relatively hyperbolic graph products of finite groups (which is a particular case of \cite[Theorem 8.33]{Qm}), follows easily from Theorem \ref{thm:specialrelhyp}. We conclude the section with the following immediate consequence of Theorem \ref{thm:specialrelhyp}:

\begin{cor}
Let $G$ be a virtually special group. Then $G$ is relatively hyperbolic if and only if there exists a finite collection of proper subgroups $\mathcal{H}$ satisfying the following conditions:
\begin{itemize}
	\item each subgroup of $\mathcal{H}$ is convex-cocompact;
	\item $\mathcal{H}$ is an almost malnormal collection;
	\item every non virtually cyclic abelian subgroup of $G$ is contained in a conjugate of some group of $\mathcal{H}$. 
\end{itemize}
If so, $G$ is hyperbolic relative to subgroups commensurable to subgroups of $\mathcal{H}$. 
\end{cor}

\subsection{Hyperbolicity relative to abelian subgroups}\label{section:RHabelian}

\noindent
It is worth noticing that the criterion provided by Theorem \ref{thm:specialrelhyp} does not provide a purely algebraic characterisation of relatively hyperbolic special groups. Indeed, the subgroups need to be convex-cocompact, but convex-cocompactness is not an algebraic property: with respect to the canonical action $\mathbb{Z}^2 \curvearrowright \mathbb{R}^2$, the cyclic subgroup generated by $(0,1)$ is convex-cocompact, whereas the same subgroup is not convex-cocompact with respect to the action $\mathbb{Z}^2 \curvearrowright \mathbb{R}^2$ defined by $(0,1) : (x,y) \mapsto (x+1,y+1)$ and $(1,0) : (x,y) \mapsto (x+1,y)$. On the other hand, we do not know whether or not a finitely generated malnormal subgroup is automatically convex-cocompact. Nevertheless, Theorem \ref{thm:specialrelhyp} provides an algebraic criterion if one restricts our attention to a collection of subgroups which we know to be convex-cocompact. In this spirit, we conclude the section with the following beautiful characterisation of virtually special groups which are hyperbolic relative to virtually abelian subgroups.

\begin{thm}\label{thm:RHspecialAbelian}
A virtually special group is hyperbolic relative to a finite collection of virtually abelian subgroups if and only if it does not contain $\mathbb{F}_2 \times \mathbb{Z}$ as a subgroup.
\end{thm}

\noindent
The following preliminary lemma will be useful during the proof of our theorem.

\begin{lemma}\label{lem:abelianInRaag}
Let $A$ be a right-angled Artin group and $a_1,p_1, \ldots, a_n,p_n \in A$ some elements. Suppose that $a_1, \ldots, a_n$ are cyclically reduced and that $p_1a_1p_1^{-1}, \ldots, p_na_np_n^{-1}$ pairwise commute. There exists some $q \in A$ such that
$$\langle p_1a_1p_1^{-1}, \ldots, p_na_np_n^{-1} \rangle= q \langle a_1, \ldots, a_n \rangle q^{-1}.$$
\end{lemma}

\noindent
We begin by a discussion written in \cite{ServatiusCent}. Let $g$ be an element in a right-angled Artin group $A(\Gamma)$. Write $g=php^{-1}$ where $h$ is the cyclic reduction of $g$, and let $\Lambda_1, \ldots, \Lambda_r$ denote the connected components of the complement of the \emph{support} of $h$, ie., the subgaph generated by the vertices of $\Gamma$ corresponding to the generators which appear in the word $h$. One can write $h$ as a product $u_1 \cdots u_k$ such that, for every $1 \leq i \leq k$, the support of $u_i$ is precisely $\Lambda_i$; this element $u_i$ belongs to a maximal cyclic subgroup whose generator, say $h_i$, is uniquely defined up to sign. So $g=ph_1^{n_1} \cdots h_r^{n_r}p^{-1}$ for some integers $n_1, \ldots, n_r \in \mathbb{Z}$. The $h_i$'s are the \emph{pure factors} of $g$. According to \cite[Centralizer Theorem]{ServatiusCent}, an element of $A(\Gamma)$ commutes with $g$ if and only if it can be written as $ph_1^{m_1} \cdots h_r^{m_r} k p^{-1}$ for some integers $m_1, \ldots, m_r \in \mathbb{Z}$ and some element $k \in A(\Gamma)$ whose support is disjoint from $\Lambda_1, \ldots, \Lambda_r$ in the complement of $\Gamma$. 

\begin{proof}[Proof of Lemma \ref{lem:abelianInRaag}.]
We argue be induction on $n$. If $n=1$, there is nothing to prove. Suppose that $n \geq 2$. By applying our induction hypothesis, we know that there exists some $q\in A$ such that
$$\langle p_1a_1p_1^{-1}, \ldots, p_{n-1}a_{n-1}p_{n-1}^{-1} \rangle= q \langle a_1, \ldots, a_{n-1} \rangle q^{-1}.$$
Let $h_1, \ldots, h_r$ denote all the pure factors which appear in $a_1, \ldots, a_{n-1}$. The Centraliser Theorem implies that $p_na_np_n^{-1}$ can be written as $q h_1^{n_1} \cdots h_r^{n_r} k q^{-1}$ where the support of $k$ is disjoint from the union of the supports of the $h_i$'s. Write $k=aha^{-1}$ where $h$ is cyclically reduced. Then
$$p_na_np_n^{-1}=(qa) \cdot h_1^{n_1} \cdots h_r^{n_r}h \cdot (qa)^{-1}$$
because $a$ commutes with all the $h_i$'s. Notice that $a_n= h_1^{n_1} \cdots h_r^{n_r}h$ by uniqueness of the cyclic reduction. Moreover, since $a$ commutes with all the pure factors which appear in $a_1, \ldots, a_{n-1}$, which are cyclically reduced, necessarily $a$ commutes with $a_1, \ldots, a_{n-1}$. Consequently,
$$\langle p_1a_1p_1^{-1}, \ldots, p_{n-1}a_{n-1}p_{n-1}^{-1} \rangle= (qa) \langle a_1, \ldots, a_{n-1} \rangle (qa)^{-1}.$$
This concludes the proof of our lemma. 
\end{proof}

\begin{proof}[Proof of Theorem \ref{thm:RHspecialAbelian}.]
Notice that, in order to show that a given group containing a special finite-index subgroup $G$ is hyperbolic relative to a finite collection of virtually abelian subgroups, it is sufficient to prove that $G$ is hyperbolic relative to a finite collection of abelian subgroups. Let $X$ be a compact special cube cube complex whose fundamental group is isomorphic to $G$. Let $\mathcal{A}$ denote the collection of maximal abelian subgroups of rank at least two. Fix a set of representatives $\mathcal{A}_0 \subset \mathcal{A}$ modulo conjugacy.

\begin{claim}\label{claim:abelianmalnormal}
If $G$ does not contain any subgroup isomorphic to $\mathbb{F}_2 \times \mathbb{Z}$, then $\mathcal{A}_0$ is a malnormal collection.
\end{claim}

\noindent
Suppose that $\mathcal{A}_0$ is not malnormal, ie., there exist two subgroups $A_1,A_2 \in \mathcal{A}_0$ and an element $g \in G$, such that either $A_1 \neq A_2$ or $A_1 = A_2$ and $g \notin A_1$, such that $A_1 \cap A_2^g \neq \{ 1 \}$. We distinguish two cases.
\begin{itemize}
	\item Suppose that $A_1 \neq A_2^g$. Since $A_1$ and $A_2^g$ are two distinct maximal abelian subgroups, there must exist two non commuting elements $k_1 \in A_1$ and $k_2 \in A_2^g$. Fix a non-trivial element $h \in A_1 \cap A_2^g$ and notice that $h$ commutes with both $k_1$ and $k_2$. By combining Theorem \ref{thm:embedRAAG} with Lemma \ref{lem:FxZinRAAG} below, we conclude that $G$ contains a subgroup isomorphic to $\mathbb{F}_2 \times \mathbb{Z}$. 
	\item Suppose that $A_1= A_2^g$. This implies that $A_1=A_2$. Let $A$ denote this common subgroup, so that $A=A^g$. Because $g \notin A$ and that $A$ is a maximal abelian subgroup, necessarily there exists some $h \in A$ such that $h$ and $g$ do not commute. By combining Theorem \ref{thm:embedRAAG} and \cite{MR634562}, we know that $\{h,g\}$ is a free basis of $\langle g,h \rangle$. On the other hand, $ghg^{-1} \in A^g=A$ must commute with $h$ since $A$ is abelian, contradiction our previous observation. Therefore, this case is impossible.
\end{itemize}
This concludes the proof of our claim. 

\begin{claim}
If $G$ does not contain any subgroup isomorphic to $\mathbb{F}_2 \times \mathbb{Z}$, then $\mathcal{A}_0$ must be finite.
\end{claim}

\noindent
Let $A \in \mathcal{A}$ be an abelian subgroup. Fix a basis $p_1a_1p_1^{-1}, \ldots, p_na_np_n^{-1} \in A$ where the spherical diagrams $a_1, \ldots, a_n$ (which do not necessarily belong to $\mathcal{D}(X,x)$) are cyclically reduced, and suppose that there exists some $1 \leq i \leq n$ such that the length of $a_i$ is greater than $\#V(X)$. Up to reordering the generators, we may suppose without loss of generality that $i=1$. By considering a path in $X$ representing the diagram $a_1$, which must pass twice through some vertex of $X$ since it has length greater than $\# V(X)$, we deduce that we can write $a_1=abc$ for some diagrams $a,b,c$ such that $b$ is a spherical diagram whose length is less than the length of $a_1$. Notice that, because $b$ is spherical, $aba^{-1}$ is a well-defined diagram. Since $\mathcal{A}_0$ is a malnormal collection according the previous claim, we deduce from the inclusion
$$\left\langle p_2a_2p_2^{-1}, \ldots, p_na_np_n^{-1} \right\rangle \subset \left\langle (pa)b(pa)^{-1}, p_2a_2p_2^{-1}, \ldots, p_na_np_n^{-1} \right\rangle \cap A$$
that 
$$A = \left\langle p_1a_1p_1^{-1}, \ldots, p_na_np_n^{-1} \right\rangle =  \left\langle (pa)b(pa)^{-1}, p_2a_2p_2^{-1}, \ldots, p_na_np_n^{-1} \right\rangle$$
where $|b|< |a_1|$. Consequently, if we choose our basis $p_1a_1p_1^{-1}, \ldots, p_na_np_n^{-1} \in A$ by minimising the sum $|a_1| + \cdots + |a_n|$, then $|a_i| \leq \# V(X)$ for every $1 \leq i \leq n$. By applying Lemma \ref{lem:abelianInRaag} (combined with Fact \ref{fact:diagvsRAAG}), we conclude that any subgroup of $\mathcal{A}$ can be written as $qFq^{-1}$ where $F$ belongs to 
$$\mathcal{F}= \left\{ \langle c_1, \ldots, c_k \rangle \left| k \geq 2, y \in X, \begin{array}{l} c_1, \ldots, c_k \in \mathcal{D}(X,y) \ \text{pairwise commute} \\ \text{and have length $<\# V(X)$} \end{array} \right. \right\}.$$
Let $t$ denote the terminus of $q$. By considering a path of minimal length in $X$ from $x$ to $t$, we find an $(x,t)$-diagram $s$ of length at most $\#V(X)$. Notice that $sq^{-1}$ belongs to $\mathcal{D}(X,x)$, and that conjugating $qFq^{-1}$ by $sq^{-1}$ gives $sFs^{-1}$. Thus, we have proved that any subgroup of $\mathcal{A}$ is conjugate to some $sFs^{-1}$ where $F \in \mathcal{F}$ and where $|s| \leq \# V(X)$. But $\mathcal{F}$ is finite and we have only finitely many choices for $s$. Consequently, $\mathcal{A}$ contains only finitely many conjugacy classes, which shows that $\mathcal{A}_0$ is indeed finite. 

\medskip \noindent
Finally, we are ready to apply Theorem \ref{thm:specialrelhyp}. Suppose that $\mathbb{F}_2 \times \mathbb{Z}$ is not a subgroup of $G$. We know from the two previous claims that $\mathcal{A}_0$ is finite malnormal collection. Moreover, it is clear that any non cyclic abelian subgroup of $G$ must be included in a conjugate of some subgroup of $\mathcal{A}_0$, and it is a general fact that abelian subgroups are convex-cocompact in groups acting geometrically on CAT(0) cube complexes, see \cite{WoodhouseFlatTorus} (although an easy direct proof is possible here). So Theorem \ref{thm:specialrelhyp} applies, so that $G$ must be hyperbolic relative to $\mathcal{A}_0$. Conversely, a group containing $\mathbb{F}_2 \times \mathbb{Z}$ as a subgroup cannot be hyperbolic relative to abelian subgroups since a subgroup which splits as a direct product of two infinite groups must be included in a peripheral subgroup (see for instance \cite[Theorem 4.19]{OsinRelativeHyp}). 
\end{proof}

\begin{lemma}\label{lem:FxZinRAAG}
Let $A$ be a right-angled Artin group and $g,h,k \in A$ be three non-trivial elements satisfying $[g,h]=[g,k]=1$ and $[h,k] \neq 1$. The subgroup $\langle g,h,k \rangle$ is isomorphic to $\mathbb{F}_2 \times \mathbb{Z}$. 
\end{lemma}

\begin{proof}
Since $g$ is clearly central in $\langle g,h,k \rangle$, necessarily $\langle g \rangle$ and $\langle h,k \rangle$ are normal subgroups in $\langle g,h,k \rangle$. Moreover, because $\langle h,k \rangle$ is a non abelian free subgroup of $A$ according to \cite{MR634562}, we know that the center of $\langle h,k \rangle$ is trivial, so that $\langle g \rangle \cap \langle h,k \rangle = \{1 \}$. A fortiori, $\langle g,h,k \rangle$ decomposes as $\langle g \rangle \times \langle h,k \rangle$. As we already observed that $\langle h,k \rangle$ is a free group, the desired conclusion follows. 
\end{proof}

\subsection{A note on elementary equivalence}\label{section:elementaryequivalence}

\noindent
Our last subsection is dedicated to the following consequence of Theorem \ref{thm:RHspecialAbelian}. We recall that two groups are \emph{elementary equivalent} if they satisfy the same first-order sequences. See \cite{SurveyLogic} and references therein for some background on the study of groups from the point of view of logic.

\begin{thm}\label{thm:RHlogic}
Let $G$ be a special group. Then $G$ is hyperbolic relative to abelian subgroups if and only if it satisfies
$$\forall a,b,c \in G \ \left( a = 1 \vee [a,b] \neq 1 \vee [a,c] \neq 1 \vee [b,c] = 1 \right).$$
Consequently, among cocompact special groups, being hyperbolic relative to abelian subgroups is preserved under elementary equivalence.
\end{thm}

\begin{proof}
If $G$ is not hyperbolic relative to abelian subgroups, it follows from Theorem~\ref{thm:RHspecialAbelian} that $G$ contains a subgroup isomorphic to $\mathbb{F}_2 \times \mathbb{Z}$. Let $a \in G$ be a generator of the cyclic factor and $b,c \in G$ two non-commuting elements of the free factor. Then $a \neq 1$, $a$ commutes with both $b$ and $c$, and $b$ and $c$ do not commute. In other words, $G$ satisfies
$$\exists a,b,c \in G \ \left( a \neq 1 \wedge [a,b]=[a,c]=1 \wedge [b,c] \neq 1 \right).$$
Conversely, suppose that $G$ satisfies the previous statement, and let $a,b,c \in G$ be the corresponding elements. Since $b$ and $c$ do not commute, it follows from \cite{MR634562} that the subgroup $\langle b,c \rangle$ is a free non-abelian subgroup. Next, notice that, because $a$ commutes with every element of $\langle b,c \rangle$ and since $\langle b,c \rangle$ is centerless, necessarily $\langle a \rangle \cap \langle b,c \rangle = \{ 1 \}$. Consequently, since $\langle a \rangle$ and $\langle b,c \rangle$ are clearly normal subgroups of $\langle a,b,c \rangle$, we conclude that
$$\langle a,b,c \rangle = \langle a \rangle \times \langle b,c \rangle \simeq \mathbb{Z} \times \mathbb{F}_2.$$
Therefore, $G$ is not hyperbolic relative to abelian subgroups. 
\end{proof}

\noindent
One possible application of Theorem \ref{thm:RHlogic} would be to reprove that \emph{limit groups} are hyperbolic relative to abelian subgroups (see for intance \cite{DahmaniCombination, alibegovic} for a proof). Indeed, we know from \cite{LimitGroups} that limit groups have the same universal theory as non-abelian free groups, so the first-order sentence of Theorem \ref{thm:RHlogic} must be satisfied by any limit group (as well as its subgroups). On the other hand, it is proved in \cite{BigWise} that limit groups are virtually special, so the desired conclusion follows. However, we emphasize that the latter result uses the fact that limit groups are hyperbolic relative to abelian subgroups, so actually we do not get an alternative proof of the relative hyperbolicity. We only know how to deduce this property from being virtually special.

\addcontentsline{toc}{section}{References}

\bibliographystyle{alpha}
{\footnotesize\bibliography{SpecialAndGraphBraidGroup}}

\end{document}